\documentclass[a4paper,10pt,reqno, english]{amsart}

\usepackage{amsmath,amssymb,amscd,amsthm,amsfonts}
\usepackage{graphicx,subfigure}
\usepackage{hyperref}
\usepackage{dsfont}
\usepackage[nobysame, alphabetic]{amsrefs}
\usepackage{tikz}
\usepackage[capitalise]{cleveref}
\usepackage{mathrsfs}

\newtheorem{theorem}{Theorem}
\newtheorem{lemma}{Lemma}

\newtheorem{corollary}{Corollary}
\newtheorem{example}{Example}

\newtheorem{definition}{Definition}

\newcommand{\ff}{\mathcal{F}}

\def\zz{\mathds{Z}}
\def\cc{\mathds{C}}
\def\rr{\mathds{R}}

\DeclareMathOperator{\St}{St}
\DeclareMathOperator{\skel}{skel}
\DeclareMathOperator{\relint}{relint}
\DeclareMathOperator{\Ind}{Ind}

\DeclareMathOperator{\suppp}{supp}
\DeclareMathOperator{\conv}{conv}
\DeclareMathOperator*{\bigast}{\mathop{\scalebox{1.5}{$\ast$}}}

\newcommand{\ZZ}{\mathds{Z}}
\newcommand{\FF}{\mathds{F}}

\newcommand{\src}{\operatorname{src}}
\newcommand{\dist}{\operatorname{dist}}

\title{KKM theorems and discrete geometry beyond matroids}

\hypersetup{
  pdftitle={KKM theorems and discrete geometry beyond matroids},
  pdfauthor={Pablo Soberon}
}

\author[Sober\'on]{Pablo Sober\'on}\address{Baruch College, City University of New York, One Bernard Baruch Way, New York, NY 10010, United States} 
\email{psoberon@gc.cuny.edu}

\thanks{
The research of P. Sober\'on is supported by NSF CAREER grant DMS-237324 and a PSC-CUNY Trad B award.}



\begin{document}

\begin{abstract}
We introduce selection structures, a topological framework that extends the role played by color classes and matroids in discrete geometry and KKM theorems. Selection structures allow us to extend classic results to genuinely non-matroidal examples, including chessboard complexes and matching complexes.

We show that several matroidal versions of classic results can be generalized to selection structures.  These include McGinnis' version of Komiya's KKMS theorem, Holmsen's version of Carath\'eodory's theorem, Kalai and Meshulam's version of Helly's theorem, and Sadovek's version of the Goodman--Pollack transversal theorem.
\end{abstract}

\maketitle

\section{Introduction}

A recurring theme in discrete geometry is the search for colorful generalizations of classic results.  Rather than working with a single set of objects, one partitions them into several color classes and asks whether information about all rainbow selections, those containing at most one element from each color, forces some property on a color class. These results often reveal a deeper combinatorial and topological framework underlying the original theorem. 

Colorful theorems are not only easy to state, but have also proved remarkably powerful. Their influence has led to a proliferation of colorful variants throughout discrete geometry and, more generally, topological combinatorics \cites{Matousek2002, Barany2021}. Prominent examples include the colorful versions of Helly's theorem, Carathéodory's theorem, and the Goodman--Pollack transversal theorem which have become cornerstones of modern discrete geometry \cite{DeLoera2019, Amenta2017, Holmsen:2017uf}. In a different direction, the Knaster--Kuratowski--Mazurkiewicz theorem and its colorful extensions play a fundamental role in fair division and envy-free allocation problems \cites{Mcginnis2024survey, Brams:1996wt, Su1999}.

For many colorful results, the structure imposed by the color classes can be replaced by the independence complex of a matroid.  This began with the work of Gil Kalai and Roy Meshulam \cite{Kalai2005}, which led to a rich collection of matroidal analogues of classical geometric and topological results \cites{Holmsen2016, Barany2017, Kim2024, McGinnis2024-arxiv, sadovek2026colorful}.  Matroidal versions of colorful results provide a much higher level of sophistication.  The guiding principle of this manuscript is that matroids can be replaced by much more flexible topological objects, which we call selection structures.  Instead of requiring exchange properties of a matroid, the relevant conditions are bounds on the connectivity and extensions of a simplicial complex.  This allows us to prove vast generalizations of the colorful and matroidal theorems mentioned above.

Intuitively, a selection structure is a pair \( \Sigma=(\mathcal{K},\mathscr{L})\) where $\mathcal{K}$ is a simplicial complex and $\mathscr{L}$ is an upwards-closed family of ``large'' sets of vertices of $\mathcal{K}$.  In the classic setting, the elements of $\mathscr{L}$ correspond to sets that have at least one point of each color class.  The complements of sets not in $\mathscr{L}$ will be used to impose the (monochromatic) conditions on the problem.  The complex $\mathcal{K}$ will be used for the (rainbow) conclusion we want to obtain, and correspond to sets that have at most one element of each color class in the classic setting.  We defer the full definition to \cref{sec:selection-structurees}.

Given a matroid $M$ we can form a selection structure by taking $\mathcal{K}$ the independence complex of $M$ and $\mathscr{L}$ the family of sets of vertices of $M$ of full rank.  The class of selection structures is strictly larger.  We impose an additional condition, $d$-admissibility, which corresponds to an upper bound of $d+1$ on the rank of $M$ in the matroid case.  We provide examples arising from graphs, chessboard complexes, and matching complexes, which are non-matroidal.  Chessboard complexes and matching complexes appear naturally in other problems in topological combinatorics, which makes these extensions appealing.

We prove selection structure versions of Komiya's extension of the KKM theorem, Helly's theorem, Carath\'eodory's, and the Goodman--Pollack transversal theorem.

\subsection{Selection structures in KKM results}
The Knaster--Kuratowski--Mazurkiewicz (KKM) theorem is a cornerstone result in topological combinatorics.  It describes the conditions to guarantee that covers of simplices have non-empty intersection.  The KKM theorem and its generalizations govern many results in envy-free distributions, discrete geometry, game theory, and optimization \cite{Mcginnis2024survey}.

\begin{theorem}[Knaster, Kuratowski, Mazurkiewicz 1929 \cite{Knaster:1929vi}]
    Let $\Delta^d$ be a $d$-dimensional simplex, with vertex set $[d+1]=\{1,\dots,d+1\}$.  Suppose $A_1,\dots, A_{d+1}$ are closed subsets of $\Delta^d$ such that for each face $\sigma$ of $\Delta^d$ we have \( \sigma \subseteq \bigcup_{j \in \sigma} A_j\).
    Then, $\bigcap_{j =1}^{d+1}A_j \neq \emptyset$.
\end{theorem}

We can take $\sigma = \Delta^d$ in the statement above, which implies that the sets $A_j$ cover $\Delta^d$.  The family $(A_1,\dots, A_{d+1})$ is called a KKM cover of $\Delta^d$.

There are numerous of extensions of the KKM theorem, which include colorful \cite{Gale1984, Frick2019}, polytopal \cites{Shih1993, Komiya1994}, sparse \cites{Soberon2022, McGinnis2024a}, balanced \cite{Shapley1973}, secretive \cite{Asada2018}, and matroidal versions  \cite{McGinnis2024-arxiv}, among others.

A far-reaching generalization of McGinnis \cite{McGinnis2024-arxiv}, which subsumes most previous results, is a matroidal version of Komiya's KKMS for polytopes.

First let us generalize KKM covers to selection structures.

\begin{definition}[\(\Sigma\)-Komiya cover]
\label{def:choice-komiya-cover}
Let $W$ be a finite set and \( \Sigma=(\mathcal K,\mathscr L) \)
be a \(d\)-admissible selection structure on \(W\).  A family of closed sets
\( \bigl(A^w_\tau
        \mid
        w\in W,\ \tau \mbox{ is a face of }P\bigr) \)
is a \(\Sigma\)-Komiya cover of \(P\) if, for every
\(W'\subseteq W\) such that
\( W\setminus W'\notin\mathscr L\) and every face $\sigma$ of $P$, we have
\[
        \sigma
        \subseteq
        \bigcup_{\tau\subseteq\sigma}
        \ \bigcup_{w\in W'}A^w_\tau .
\]
\end{definition}

We prove the following selection structure generalization of McGinnis' theorem.

\begin{theorem}[Selection-structure Komiya theorem]
\label{thm:selection-structure-komiya}
Let \(P\) be a \(d\)-dimensional polytope, $p \in P$ and let $W$ be a finite set.  Let \( \Sigma=(\mathcal{K},\mathscr{L}) \)
be a \(d\)-admissible selection structure on \(W\).  For every
\(w\in W\) and every face \(\tau\) of \(P\), let \( A^w_\tau\subseteq P\)
be a closed set, and let
\(y^w_\tau\in\tau\) be a point.

If \( \bigl(A^w_\tau \mid w\in W,\ \tau\mbox{ is a face of }P\bigr) \)
is a \(\Sigma\)-Komiya cover of \(P\), then there exist a containment-maximal face
\( I\in\mathcal K \)
and faces
\( \tau_w\) of \( P \) for each \( w\in I\) such that
\( p\in
        \operatorname{conv}\{y^w_{\tau_w}:w\in I\} \)
and
\[
        \bigcap_{w\in I} A^w_{\tau_w}\ne\emptyset .
\]
\end{theorem}

When we specialize our proof of \cref{thm:selection-structure-komiya} to prove McGinnis' result, we obtain a new proof of his theorem.  Our proof reduces the problem to the construction of a face-preserving map $Y:P \to P$.  This implies $Y|_{\partial P}$ is not null-homotopic, which is the key step in the proof.

\subsection{Selection structures in combinatorial geometry}

Helly's theorem \cite{Helly:1923wr} and Carath\'eodory's theorem \cite{Caratheodory1907} are some of the most important results in combinatorial geometry.  Their colorful versions by Lov\'asz and B\'ar\'any, respectively, pioneered the use of colorful theorems in discrete geometry.  Both results were published in the same paper by B\'ar\'any.

\begin{theorem}[Colorful Helly, Lov\'asz 1982 \cite{Barany1982}]\label{thm:colorful-helly}
Let $d$ be a positive integer and $\ff_1,\dots, \ff_{d+1}$ be finite families of convex sets in $\rr^d$.  If every $(d+1)$-tuple of sets in different $\ff_i$ has non-empty intersection, then some $\ff_i$ must have non-empty intersection.
\end{theorem}

\begin{theorem}[Colorful Carath\'eodory, B\'ar\'any 1982 \cite{Barany1982}]\label{thm:colorful-caratheodory}
    Let $d$ be a positive integer and $X_1,\dots,X_{d+1}$ be subsets of $\rr^d$.  If $0 \in \conv X_i$ for all $i$, then there is a choice $x_1 \in X_1,\dots, x_{d+1}\in X_{d+1}$ such that $0 \in \conv\{x_1,\dots,x_{d+1}\}$.
\end{theorem}

As mentioned earlier in the introduction, Kalai and Meshulam proved a matroidal topological Helly theorem \cite{Kalai2005}, which has recently been generalized by Kim and Lew \cite{Kim2024}.  A matroidal colorful Carath\'eodory is a direct consequence of a general result of Holmsen \cite{Holmsen2016}.  McGinnis also shows that his matroidal Komiya KKMS theorem implies the matroidal colorful Carath\'eodory \cite{McGinnis2024-arxiv}.  Blagojevi\'c recently proved a generalization of the matroidal Carath\'eodory theorem \cite{blagojevic2025colorful}.  Neither of Blagojevi\'c's result nor \cref{thm:selection-structure-colorful-caratheodory} seem to imply the other.

We prove selection structure versions of both theorems, generalizing the matroidal and colorful extensions.

\begin{theorem}[Selection-structure Carathéodory]
\label{thm:selection-structure-colorful-caratheodory}
Let \(d\) be a positive integer, \(W\) be a finite set, and
\( \Sigma=(\mathcal K,\mathscr L) \)
be a \(d\)-admissible selection structure on \(W\).
For every \(w\in W\), let \(X_w\subset \rr^d\) be a nonempty finite set.

Suppose that, for every \(W'\subseteq W\) such that \( W\setminus W'\notin\mathscr L \), we have
\[
        0\in
        \operatorname{conv}
        \Bigl(\bigcup_{w\in W'}X_w\Bigr).
\]
Then there exist a set \(I\in\mathcal K\) and points
\(x_w\in X_w\) for all \(w\in I,\)
such that \({0\in \operatorname{conv}\{x_w:w\in I\}}\).
\end{theorem}

We state our selection structure Helly theorem in terms of the contrapositive version (i.e., guaranteeing empty intersections).

\begin{theorem}[Selection-structure Helly]\label{thm:selection-structure-helly}
    Let $W$ be a finite set and $\mathcal{F}=\{F_w: w \in W\}$ be a family of convex sets in $\rr^d$ indexed by $W$.  Let $\Sigma=(\mathcal{K}, \mathscr{L})$ be a $d$-admissible selection structure on $W$.  Suppose that for every $W' \subseteq W$ such that $W \setminus W' \not\in \mathscr{L}$ we have $\bigcap_{w\in  W'} F_w = \emptyset$.  Then, there exists $I \in \mathcal{K}$ such that $\bigcap_{w \in I}F_w = \emptyset$.
\end{theorem}

We give two proofs of the selection structure Carath\'eodory theorem.  We also give two proofs of the selection structure Helly theorem.  One gives a very simple proof of \cref{thm:selection-structure-helly} and ,when applied to the particular case of a matroid, gives a new proof of the Kalai--Meshulam Helly theorem for convex sets.  The second one generalizes the full Kalai--Meshulam topological Helly theorem for good covers in $\rr^d$.

We also prove selection structure versions of the Goodman--Pollack theorem on transversals to families of convex sets.

The Goodman--Pollack transversal problem asks for conditions that guarantee the
existence of an affine \(k\)-dimensional transversal to a finite family of convex
sets in \(\rr^d\).  The case $k=0$ is Helly's theorem.

The classic result is the solution to the case $k=d-1$ by Goodman and Pollack \cites{Goodman1988, Pollack1990, Wenger1990}.  There is a long history of variations and extensions of the Goodman--Pollack result \cites{Goodman1993, Holmsen:2017uf}.  Recently, McGinnis and Sadovek fully characterized the general case \cite{McGinnis2026}, achieving the culmination of a decades-long effort in the area.  Shortly after that, Sadovek proved a matroidal version of the McGinnis--Sadovek theorem \cite{sadovek2026colorful}, extending the previous colorful versions of Goodman--Pollack \cites{Holmsen2022, Cheong2024}.

The Goodman-–Pollack transversal problem is the most delicate application of selection structures.  We show that the matroidal part of Sadovek’s argument can be transferred to selection structures. This gives transversal theorems indexed by non-matroidal objects, including chessboard and matching complexes.

We present the selection structure version in \cref{thm:selection-structure-sadovek}.  Since the characterization by McGinnis and Sadovek for the existence of $k$-dimensional affine transversals involves a delicate condition, we defer the definitions and statement of the result to \cref{sec:selection-structure-goodman-pollack}.  Just like the McGinnis-Sadovek theorem implies Helly's theorem, \cref{thm:selection-structure-sadovek} implies \cref{thm:selection-structure-helly}.

We also show that \cref{thm:selection-structure-sadovek} implies a selection structure version of Dolnikov's theorem on $k$-transversals to convex sets \cite{Dolnikov1992}.  This is the central argument for Dolnikov's proof of the central transversal theorem \cites{Dolnikov1992, Zivaljevic1990}.  We state here the version for $\rr^d$, while the full statement in \cref{sec:selection-structure-goodman-pollack} also generalizes to $\cc^d$.

\begin{theorem}[Selection-structure Dolnikov transversal theorem]
\label{thm:selection-structure-dolnikov-real}
Let \(0\le r\le k<d\) be integers, and \(
        N=(d-k)(r+1)\).
Let \(W\) be a finite set, let \( W=W_1\sqcup\cdots\sqcup W_{r+1}\)
be a partition, and let \( \Sigma=(\mathcal K,\mathscr L) \)
be an \(N\)-admissible selection structure on \(W\).
For each \(w\in W\), let \(F_w\subseteq \rr^d\) be a nonempty compact
convex set.

Assume that for every \(i=1,\dots,r+1\) and every face
\(I\in\mathcal K\) such that
\( I\subseteq W_i\) and \( |I|\le d-k+1\),
we have
\( \bigcap_{w\in I}F_w\neq\emptyset\).  Then there exists \(W'\subseteq W\) such that the family \( \{F_w:w\in W'\}\)
has an affine \(k\)-transversal and
\( W\setminus W'\notin\mathscr L\).
\end{theorem}

We define selection structures properly in \cref{sec:selection-structurees}, and introduce new examples in \cref{sec:chessboard-matching-choice}.  The first results we prove are the extensions of the KKM theorem in \cref{thm:selection-structure-komiya}.  Then, we prove our selection-structure versions of Carath\'eodory and Helly in \cref{sec:helly-and-caratheodory}, before proving our results related to the Goodman--Pollack problem in \cref{sec:selection-structure-goodman-pollack}.  We conclude with remarks and open problems in \cref{sec:remarks}.

\section{Selection structures}\label{sec:selection-structurees}

In this section we will define selection structures, the main new object used in the results of this manuscript.

Let \(W\) be a finite set.  We call the elements of $W$ the sources.  Let $\mathcal{K}$ be a simplicial complex on \(W\) and $\mathcal{B}=\operatorname{Facets}(\mathcal{K})$ be the antichain of containment-maximal faces of $\mathcal{K}$.  We use the notation $\mathcal{B}$ since they correspond to bases if $\mathcal{K}$ is the independence complex of a matroid.  For our proofs, we require some flexibility to expand $\mathcal{K}$.  We call the following construction a \textit{parallel expansion}.  For a set $U \subseteq W$, let $\mathcal{K}[U] = \{ \sigma \cap U: \sigma \in \mathcal{K}\}$.

Let $D=\{D_w: w \in W\}$ be a family of nonempty finite sets indexed by $W$, which we call fibers. Let \(U\subseteq W\).  The iterated join $\displaystyle \bigast_{w \in W} D_w$ can be interpreted as the simplices that have at most one vertex from each source.  Define
\[
        \mathcal{K}(U;D) = \bigcup_{F \in \mathcal{K}[U]} \left( \bigast_{w \in F} D_w  \right).
\]

In other words, we take the simplices from the iterated join such that the sources they use form a face of $\mathcal{K}[U]$.  For a face $\eta$ of $\mathcal{K}(U;D)$, we denote by $\src(\eta) \in \mathcal{K}$ the set of sources it uses.

Similarly, we define
\[
        \mathcal B(U;D)
        =
        \{\beta\in \mathcal K(U;D):
        \operatorname{src}(\beta)\in\mathcal B\}.
\]

\begin{definition}[\(d\)-admissible selection structure]
\label{def:d-admissible-choice-datum}
Let \(d\ge 0\).  A \(d\)-admissible selection structure on \(W\) is a pair
\( \Sigma=(\mathcal{K},\mathscr{L})\),
where \(\mathcal{K}\) is a simplicial complex, and
\(\mathscr{L}\subseteq 2^W\) is a non-empty upwards-closed family, such that the
following two conditions hold.

\begin{enumerate}
\item[{(Connectedness)}]
For every \(U\in\mathscr{L}\) and every family of finite sets \({D=\{D_w: D_w \neq \emptyset, w \in W}\}\), the space
\( \mathcal{K}(U;D) \)
is \((d-1)\)-connected.

\item[{(Completion)}]
For every $U \in \mathscr{L}$, we have $\operatorname{Facets}(\mathcal{K}[U]) \subseteq \operatorname{Facets}(\mathcal{K})$.

\end{enumerate}
\end{definition}

The completion axiom states that every induced complex $\mathcal{K}[U]$ cannot have new containment-maximal faces.  Another way to state this is to consider $\mathcal{B}=\operatorname{Facets}(\mathcal{K})$.  Then, for every \(U\in\mathscr{L}\), every family of finite sets \(D=\{D_w: D_w \neq \emptyset, w \in W\}\), and every face
\( \eta\in\mathcal{K}(U;D)\),
there exists \(\beta\in\mathcal B(U;D)\)
with \(\eta\subseteq \beta\).  The requirement that the sets $D_w$ are finite is not strictly necessary.  The examples of selection structures we present still satisfy the connectedness condition when we introduce infinite fibers.  However, since the proofs of our results only use finite fibers, we impose this condition.

One key example is given by matroids, which give $d$-admissible selection structures.

\begin{example}[Matroids]\label{example:matroids}
Let \(M\) be a matroid of rank greater than or equal to \(d+1\) with vertex set \(W\) and let $r_M$ be its rank function.  Let $\operatorname{Ind}(M)^{(d)}$ be the $d$-skeleton of its independence complex.  We define
\[
        \mathcal{K}=\operatorname{Ind}(M)^{(d)},
        \qquad
        \mathscr{L}_{M,d}
        =
        \{U\subseteq W:r_M(U)\ge d+1\}.
\]
Then
\(\Sigma_M
        =
        (\operatorname{Ind}(M)^{(d)},\mathscr{L}_{M,d}) \)
is a \(d\)-admissible selection structure.
\end{example}

\begin{proof}
First, \(\mathscr L_{M,d}\) is nonempty, since \(W\in \mathcal L_{M,d}\), and it is upward
closed by monotonicity of the matroid rank function.

We verify the connectedness condition. Fix \(U\in \mathcal L_{M,d}\) and a family of
nonempty finite fibers \(D=\{D_w:w\in W\}\).
Let
\[
        \widetilde U=\{(w,a):w\in U,\ a\in D_w\},
\]
and let \(\pi:\widetilde U\to U\) be the projection \(\pi(w,a)=w\). Consider the
parallel extension \(M_{U,D}\) of the restricted matroid \(M|U\) on the ground set
\(\widetilde U\): a set \(I\subseteq \widetilde U\) is independent in \(M_{U,D}\) if and
only if \(\pi\) is injective on \(I\) and \(\pi(I)\) is independent in \(M|U\). The
elements in each fiber \(D_w\) form a parallel class, and \(M_{U,D}\) has rank
\(r_M(U)\).

Since \(U\in\mathscr L_{M,d}\), we have \(r_M(U)\ge d+1\). Let \(T_{d+1}(M_{U,D})\)
denote the rank-\((d+1)\) truncation of \(M_{U,D}\). Its independent sets are
precisely the independent sets of \(M_{U,D}\) of cardinality at most \(d+1\).
Therefore
\[
        \mathcal{K}(U;D)=\operatorname{Ind}\bigl(T_{d+1}(M_{U,D})\bigr).
\]

The independence complex of a matroid is shellable \cite{Bjoerner1992}. Since
\(T_{d+1}(M_{U,D})\) has rank \(d+1\), its independence complex is a pure
\(d\)-dimensional shellable complex. Hence it is \((d-1)\)-connected. This means that
\(\mathcal{K}(U;D)\) is \((d-1)\)-connected.

The completion axiom is a direct consequence of the matroid basis extension property.
\end{proof}

A specialization of the example above is the following, which is related to a wide number of applications in discrete geometry.

\begin{example}[colorful choices]\label{ex:colorful}
    Suppose the elements of $W$ each have one of $d+1$ possible colors.  Consider $\mathcal{K}$ the simplicial complex of rainbow sets (the sets that have at most one element of each color), and $\mathscr{L}$ the family of sets that have at least one element from each color.  These correspond to the elements of \cref{example:matroids} with a partition matroid of rank $d+1$, and therefore form a $d$-admissible selection structure.
\end{example}

\cref{ex:colorful} recovers the classic colorful versions of the theorems we mentioned.  Now let us show examples of selection structures which do not come from matroids, to show that selection structures are indeed more general.

\begin{example}[Selection structures from graphs]\label{ex:selection-structure-graph}
    Let $H$ be a connected graph without isolated vertices, with vertex set $W$ and edge set $E$.  Consider $\mathcal{K}$ the graph $H$ as a simplicial complex.  Let
    \begin{align*}
        \mathscr{L} & = \left\{W':\begin{gathered} \mbox{the graph induced by } W' \mbox{ is connected,} \\\mbox{ every vertex in $W \setminus W'$ has an edge to $W'$, and} \\ W' \mbox{has at least two vertices}\end{gathered} \right\}.
    \end{align*}
    The pair \( \Sigma=(\mathcal{K},\mathscr{L})\) is $1$-admissible.  The complex $\mathcal{K}(W;D)$ is simply the graph obtained by replacing any vertex $w$ by an independent set $D_w$, and for any two adjacent vertices $u,v$ we include the full bipartite graph with components $D_u, D_v$.
\end{example}

In \cref{sec:improving-selection-structures} we discuss why this example leads to applications that cannot be recovered by existing matroid versions.  In general, the difficult part of showing that a selection structure is $d$-admissible is to guarantee that, once the set $D$ is involved, the parallel expansions have high enough connectivity.

\section{Chessboard and matching selection structures}
\label{sec:chessboard-matching-choice}

We now give two systematic sources of selection structures that do not come from matroids:
chessboard complexes and matching complexes.  Chessboard complexes have been studied due to their applications in other geometric problems, such as the colorful versions of Tverberg's theorem \cites{Bjoerner1994, Vrecica2011}.

For a simplicial complex $K$, let $K^{(d)}$ be its $d$-skeleton.

\subsection{Parallel expansions of matching complexes}

Let $H$ be a finite graph or hypergraph.  We denote by
\[
        \operatorname{Match}(H)
\]
its matching complex: the vertices are the edges of $H$, and the faces
are pairwise disjoint collections of edges.  If $H$ is a graph with edge
set $W$ and $U\subseteq W$, we write $H[U]$ for the subgraph with edge
set $U$.

If $D=\{D_w:w\in U\}$ is a family of nonempty sets, let $H[U]^D$ be the
multi-graph or multi-hypergraph obtained from $H[U]$ by replacing each
edge $w\in W$ by a parallel class indexed by $D_w$.

\begin{lemma}
\label{lem:parallel-expansion-matching}
Let $H$ be a graph or hypergraph with edge set $W$, and let \(\mathcal{K}=\operatorname{Match}(H)\).
Then, for every $U\subseteq W$ and every family of nonempty sets
$D=\{D_w:w\in U\}$, there is a natural isomorphism
\[
        \mathcal{K}(U;D)\cong \operatorname{Match}(H[U]^D).
\]
If $\mathcal{K}=\operatorname{Match}(H)^{(d)}$ is the $d$-skeleton, then
\( \mathcal{K}(U;D)\cong \operatorname{Match}(H[U]^D)^{(d)}\).
\end{lemma}

\begin{proof}
A face of $\mathcal{K}(U;D)$ is a set of vertices $(w,\alpha)$, with
$\alpha\in D_w$, such that no edge $w$ is used twice and the set of
edges is a matching in $H[U]$.  This is exactly a matching in the
parallel expansion $H[U]^D$.
\end{proof}

The chessboard complex $\Delta_{a,b}$ is a simplicial complex where the vertices are the squares of an $a \times b$ grid and the faces are the sets of squares that do not repeat rows nor columns.  In other words, they correspond to non-taking placements of rooks.  Chessboard complexes are a special case of matching complexes, as ${\Delta_{a,b} = \operatorname{Match}(K_{a,b})}$, where $K_{a,b}$ denotes the complete bipartite graph with components of size $a$ and $b$.

The connectivity of chessboard complexes was found by Bj\"orner, Lov\'asz, Vre\'cica, and \v{Z}ivaljevi\'c \cite{Bjoerner1994}.  Our main technical lemma in this section is that the connectivity bound by Bj\"orner et al. is stable under arbitrary addition of extra edges, including
parallel edges.

For $a,b\geq 1$, consider
\[
        \nu(a,b)
        =
        \min\left\{
            a,\,
            b,\,
            \left\lfloor \frac{a+b+1}{3}\right\rfloor
        \right\}.
\]

\begin{lemma}[Bj\"orner, Lov\'asz, Vre\'cica, \v{Z}ivaljevi\'c 1994 \cite{Bjoerner1994}]\label{lem:chessboard-connected}
    Let $a,b$ be positive integers.  The chessboard complex $\Delta_{a,b}$ is
$(\nu(a,b)-2)$-connected.
\end{lemma}

%----

For our construction, we will use a special kind of subsets of $\Delta_{M,N}$, which we call $(a,b,d)$-stairs.

\begin{definition}
    Let $(a,b,d)$ be integers such that $\nu(a,b) \ge d+1$, $a \le 2d+1$, and $b \le 2d+1$.  The $(a,b,d+1)$-stair is the set of squares $(i,j) \in [a]\times[b]$ such that
    \[
    a-(2d+1) \le j-i \le (2d+1)-b.
    \]
\end{definition}

We show two examples of $(a,b,d)$-stairs in \cref{fig:stairs}.  Notice that the facets of $\Delta_{2,3}$ form a graph, a cycle of length $6$.  The facets of the $(2,3,2)$-stair form a path of length $4$.  This will be relevant in \cref{sec:remarks}.

\begin{figure}
    \centering
    \includegraphics[width=0.8\linewidth]{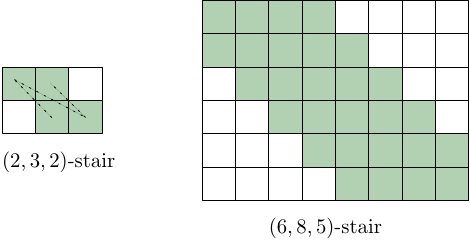}
    \caption{An example of the $(2,3,2)$-stair and the $(6,8,5)$-stair.  Notice that the chessboard complex induced by the squares of the $(2,3,2)$-stair is a path with four vertices, shown by the dotted lines.}
    \label{fig:stairs}
\end{figure}

Ziegler proved that if a subcomplex of a chessboard complex contains an isomorphic copy of an $(a,b,d)$-stair, then its $d$-skeleton is vertex-decomposable \cite{Ziegler1994}.

Recall that we can define vertex-decomposability is defined recursively as follows.  A simplicial complex $K$ is vertex-decomposable if there is a vertex $v$ (called a \textit{shedding vertex}) such that:
\begin{itemize}
    \item $K \setminus v$ is vertex-decomposable,
    \item the link $\operatorname{lk}(v)$ is vertex-decomposable, and
    \item no facet of $\operatorname{lk}(v)$ is a facet of $K\setminus v$.
\end{itemize}

To use this in the context of selection structures, we need the following lemma.

\begin{lemma}\label{lem:vertex-decomposable}
    Let $\mathcal{K}$ be a vertex-decomposable simplicial complex of dimension $d$.  Then every finite parallel expansion of $\mathcal{K}$ is again a pure $d$-dimensional vertex-decomposable complex.  In particular, it is $(d-1)$-connected.
\end{lemma}

\begin{proof}
We add the additional parallel vertices one at a time.  Adding a new vertex $v'$ which is a copy of $v$ turns our complex $\mathcal{K}$ into
\[
\mathcal{K'} = \mathcal{K} \cup (v' * \operatorname{lk}_{\mathcal{K}}(v))
\]
Since in a vertex-decomposable simplicial complex the link of any face (in particular a vertex) is also vertex-decomposable \cites{Provan1980, Biermann2015}, we know that $\operatorname{lk}_{\mathcal{K}}(v)$ is vertex decomposable and pure $(d-1)$-dimensional.  We also have that the deletion of $v'$ in $\mathcal{K}'$ is $\mathcal{K}$.  Since the facets of $\mathcal{K}$ are $d$-dimensional and the facets of $\operatorname{lk}_{\mathcal{K}}(v)$ are $(d-1)$-dimensional, we have that $v'$ is a shedding vertex of $\mathcal{K}'$ as long as $\mathcal{K}$ was vertex decomposable.  If we proceed one vertex at a time, we obtain the desired conclusion.
\end{proof}

\subsection{Chessboard selection structures}

Let $M,N$ be positive integers and let \( W=[M]\times [N]\).
We identify $W$ with the edge set of $K_{M,N}$.  Recall that we call the elements of $W$ the sources.

\begin{theorem}[Chessboard selection structures]
\label{thm:chessboard-selection-structures}
Let $ M,N$ be positive integers such that \(\nu(M,N)\ge d+1\).  Consider
\begin{align*}
     \mathcal{K}&=\Delta_{M,N}^{(d)} \\
     \mathscr{L}^{\operatorname{ch}}_{M,N,d}
        &=
        \left\{
            U\subseteq [M]\times [N]:
            \substack{\mbox{There exist }a,b \mbox{ such that }\nu(a,b) \ge d+1 \mbox{ and } a \le 2d+1,\, b \le 2d+1\\
             \mbox{ and a subset }U' \subseteq U \mbox{that is isomorphic to the }(a,b,d+1)\mbox{-stair}
            }
        \right\}.
\end{align*}
Then
\( \Sigma^{\operatorname{ch}}_{M,N,d}
        =
        \bigl(\mathcal{K},\,\, \mathscr{L}^{\operatorname{ch}}_{M,N,d}\bigr) \)
is a $d$-admissible selection structure.
\end{theorem}
We can think of $\mathcal{K}$ as the placements of at most $d+1$ non-taking rooks.  We also have  $U\in\mathscr{L}^{\operatorname{ch}}_{M,N,d}$ if and only if, after some permutation of rows and columns, the selected squares contain
the $(a,b,d+1)$-stair.  The family $\mathscr{L}^{\operatorname{ch}}_{M,N,d}$ is
upward closed.

\begin{proof}
Let $U\in\mathscr{L}^{\operatorname{ch}}_{M,N,d}$, $W = [M] \times [N]$, and let
$D=\{D_w:w\in U\}$ be a family of nonempty finite sets. 

By Ziegler's results \cite{Ziegler1994}, for each $U \in  \mathscr{L}^{\operatorname{ch}}_{M,N,d}$ the complex $\mathcal{K}[U]$ is pure $d$-dimensional and vertex-decomposable.  Then, by \cref{lem:vertex-decomposable}, the complex $\mathcal{K}(U;D)$ is $(d-1)$-connected.

Finally, since $\mathcal{K}[U]$ is pure $d$-dimensional, its facets have $d+1$ vertices, which means they are facets of $\mathcal{K}$. 
\end{proof}

To exemplify the kind of results that these selection structures imply, consider the direct application to colorful Carath\'eodory, \cref{thm:selection-structure-colorful-caratheodory}.

\begin{corollary}[Chessboard Carath\'eodory]\label{cor:chessboard-caratheodory}
    Let $M,N,d$ be positive integers such that $\nu(M,N) \ge d+1$ and $W = [M] \times [N]$.  For each $w = (u,v) \in W$, let $x_w$ be a point in $\rr^d$.

    We know that each set $W' \subseteq W$ that intersects every subset of $W$ isomorphic to some $(a,b,d+1)$-stair for some $a,b \le 2d+1$ with $\nu(a,b) \ge d+1$, we have $0 \in \conv \{x_w : w\in W'\}$.  Then, we can find $d+1$ points in the set $\{x_w: w \in W\}$ that correspond to a non-taking placement of rooks and such that their convex hull contains the origin.
\end{corollary}

\begin{proof}
    Apply \cref{thm:selection-structure-colorful-caratheodory} with $X_w = \{x_w\}$ for all $w \in W$ and $\Sigma$ the selection structure from \cref{thm:chessboard-selection-structures}.
\end{proof}

\subsection{Matching-complex selection structures}

There is an analogous construction for matching complexes of complete
uniform hypergraphs.  Let $K_N^k$ be the complete $k$-uniform hypergraph
on $N$ vertices.  Its matching complex is denoted \( \operatorname{Match}(K_N^k)\).
For $n\geq k$, put
\[
        \nu_k(n)
        =
        \left\lfloor
            \frac{n+2k-3}{2k-1}
        \right\rfloor .
\]
Again, Björner et al. computed the connectivity of matching complexes and showed that
$\operatorname{Match}(K_n^k)$ is $(\nu_k(n)-2)$-connected.

\begin{lemma}[Bj\"orner, Lov\'asz, Vre\'cica, \v{Z}ivaljevi\'c 1994 \cite{Bjoerner1994}]\label{lem:matching-connected}
    Let $n \ge k$ be positive integers.  Then, the complex $\operatorname{Match}(K_n^k)$ is $(\nu_k(n)-2)$-connected.
\end{lemma}

\begin{lemma}
\label{lem:robust-hypergraph-connectivity}
Let $H$ be a finite $k$-uniform multi-hypergraph.  Suppose that $H$
contains a copy of $K_a^k$.  Then \( \operatorname{Match}(H)\)
is $(\nu_k(a)-2)$-connected.
\end{lemma}

\begin{proof}
The proof is based on an iterative cone-attachment argument.  Start with a fixed
copy of $K_a^k$, whose matching complex is $(\nu_k(a)-2)$-connected by
\cite{Bjoerner1994}.  Add all remaining hyperedges one at a time.

When a new hyperedge $e$ is added to $H_0$, we have
\[
        \operatorname{Match}(H_0\cup\{e\})
        =
        \operatorname{Match}(H_0)
        \cup
        \left(e*\operatorname{Match}(H_0-V(e))\right),
\]
and the intersection is $\operatorname{Match}(H_0-V(e))$.  If $e$ meets
the copy of $K^k_a$ in $s$ vertices, where
$0\leq s\leq k$, then $H_0-V(e)$ contains a copy of
$K_{a-s}^k$.  Since $s\leq k<2k-1$, we have
\(\nu_k(a-s)\geq \nu_k(a)-1\).

By induction, the attaching subcomplex is at least
$(\nu_k(a)-3)$-connected.  Attaching a cone along such a subcomplex
preserves $(\nu_k(a)-2)$-connectivity.
\end{proof}

Now we are ready to define selection structures for complete hypergraphs.

\begin{theorem}[Matching selection structures]
\label{thm:matching-selection-structures}
Let $N,a,k,d$ be positive integers such that
\(\nu_k(a)\geq d+1\) and \(N \ge a\geq k(d+1)\).
\begin{align*}
    W&=E(K_N^k) \\
    \mathcal{K}&=\operatorname{Match}(K_N^k)^{(d)} \\
    \mathscr{L}_a
        & =
        \left\{
            U\subseteq W:
            \text{there is an }a\text{-element set }A\subseteq [N]
            \text{ such that }E(K_A^k)\subseteq U
        \right\}.
\end{align*}
Then
\(
        \Sigma=
            \bigl( \mathcal{K},
            \mathscr{L}_a
        \bigr)
\)
is a $d$-admissible selection structure.
\end{theorem}

\begin{proof}
The connectivity condition follows from
Lemma~\ref{lem:parallel-expansion-matching} and
Lemma~\ref{lem:robust-hypergraph-connectivity} as follows.

Let $D=\{D_w: w \in W\}$ be a set of finite fibers, and let $U \in \mathscr{L}_a$.  Since $U$ contains a copy of $K^k_a$, then $\mathcal{K}(U;D)$ does too.  By Lemma~\ref{lem:parallel-expansion-matching} and Lemma~\ref{lem:robust-hypergraph-connectivity}, the complex $\mathcal{K}(U;D)$ is at least $(d-1)$-connected.

For completion, let $\eta$ be a facet of $\mathcal{K}[U]$.  Since $U$ contained a complete hypergraph with $k(d+1)$ vertices, we know that $\eta$ must contain $d+1$ edges.  Therefore, $\eta$ was already a facet of $\mathcal{K}$.
\end{proof}

For $k=2$, this gives selection structures from ordinary matching complexes
of complete graphs.  In this case
\[
        \nu_2(a)=\left\lfloor\frac{a+1}{3}\right\rfloor .
\]

\section{Proof of \cref{thm:selection-structure-komiya}}

We use the following well known lemma.

\begin{lemma}\label{cor:face-preserving}
    Let $P$ be a polytope and $f: P \to P$ be a continuous map.  Assume that for each proper face $\sigma$ of $P$, we have $f(\sigma) \subseteq \sigma$.  Then, $f$ is surjective.
\end{lemma}

Let us give a proof of \cref{cor:face-preserving} for completeness.

\begin{proof}[Proof of \cref{cor:face-preserving}]
    Let $g: \partial P \to \partial P$ be the identity.  Note that the map $f$ sends the boundary of $P$ to the boundary of $P$.  Since the faces of $P$ are convex, for any face $\sigma$ of $P$, $x \in \sigma$, and $t \in [0,1]$, we have $t f(x) + (1-t)g(x) \in \sigma$.  Therefore
    \begin{align*}
        H: \partial P \times [0,1] & \to \partial P \\
        (x,t) & \mapsto t f(x) + (1-t)g(x)
    \end{align*}
    is well defined and the maps $g$ and $f|_{\partial P}$ are homotopic.  Moreover since $\partial P \cong S^{d-1}$, the map $f|_{\partial P}$ is of degree $1$, and therefore $f$ must be surjective.  
\end{proof}

Before proving \cref{thm:selection-structure-komiya}, we present a discrete version.  The KKM theorem has a discrete equivalent version, Sperner's lemma \cite{Sperner1928}.  The definition and result below is simply the analogue of Sperner's lemma to our setting.  For a triangulation $\mathcal{T}$ of a polytope $P$, let $V(\mathcal{T})$ be its set of vertices.

\begin{definition}[\(\Sigma\)-Sperner--Shapley labelling]
\label{def:choice-sperner-labelling}
Let $W$ be a finite set and \( \Sigma=(\mathcal{K},\mathscr{L}) \)
be a \(d\)-admissible selection structure on \(W\).  Let $P$ be a polytope and $\mathcal{T}$ be a face-refining triangulation of $P$.  A family of sets of vertices of $\mathcal{T}$
\( \bigl(A^w_\tau
        \mid
        w\in W,\ \tau \mbox{ is a face of }P\bigr) \)
is a \(\Sigma\)-Sperner--Shapley labelling of \(\mathcal{T}\) if, for every
\(W'\subseteq W\) such that
\( W\setminus W'\notin\mathscr{L}\) and every face $\sigma$ of $P$, we have
\[
        \sigma \cap V(\mathcal{T})
        \subseteq
        \bigcup_{\tau\subseteq\sigma}
        \ \bigcup_{w\in W'}A^w_\tau .
\]
\end{definition}

\begin{theorem}[Selection structure Sperner--Shapley theorem]
\label{thm:selection-structure-shapley}
Let \(P\) be a \(d\)-dimensional polytope, $\mathcal{T}$ be a face-refining triangulation of $P$, $p \in P$, and $W$ be a finite set.  Let \( \Sigma=(\mathcal{K},\mathscr{L}) \)
be a \(d\)-admissible selection structure on \(W\).  For every
\(w\in W\) and every face \(\tau\) of \(P\), let \( A^w_\tau\subseteq V(\mathcal{T})\), and let
\(y^w_\tau\in\tau\) be a point.

If \( \bigl(A^w_\tau \mid w\in W,\ \tau\mbox{ is a face of }P\bigr) \)
is a \(\Sigma\)-Sperner--Shapley labeling of \(\mathcal{T}\), then there exist a containment-maximal face
\( I\in\mathcal K \)
and faces
\( \tau_w\) of \( P \) for each \( w\in I\) such that
\( p\in
        \operatorname{conv}\{y^w_{\tau_w}:w\in I\} \)
and such that there exists a simplex of $\mathcal{T}$ whose vertices have labels in every set $A^w_{\tau_w}$ for $w\in I$.
\end{theorem}

\begin{proof}[Proof of \cref{thm:selection-structure-shapley}]

Let us first describe the idea behind the proof.  Let $\ff (P)$ be the lattice of non-empty faces of $P$.  The goal will be to construct a map ${\psi: P \to |\mathcal{K}(W;\mathcal{F}(P))|}$, the construction requires the $d$-admissibility of $\Sigma$ and the Sperner--Shapley condition.  Using $\psi$, we then construct a map $Y: P \to P$ that is face-preserving: $Y(\sigma) \subseteq \sigma$ for all faces $\sigma$ of $P$.  Then, a standard consequence of \cref{cor:face-preserving} is that $Y$ is surjective, which implies that $p$ is in its image.  We use $Y^{-1}(p)$ to conclude.

To simplify the notation in the proof, we use $\alpha = (w,\tau)$ to denote the element of $W \times \ff(P)$.  We extend this to $y_{\alpha} = y^w_{\tau}$ and $A_{\alpha} = A^w_{\tau}$.  We use $\src(\alpha) = w$ to recover the component from $W$.

For any point $x \in P$, let $\suppp_P(x)$ be the minimal face of $P$ that contains $x$.  For any face $\tau$ of $P$, let $V(\tau)$ be its set of vertices.  For a vertex \(v\in V(\mathcal T)\) and \(w\in W\), put
\[
        D_w(v)
        =
        \{\tau \mbox{ a face of } P:
        \tau\subseteq \operatorname{supp}_P(v)
        \text{ and }
        v\in A^w_\tau\}.
\]
Let \( U(v)=\{w\in W:D_w(v)\ne\emptyset\}\).

We first claim that \(U(v)\in\mathscr{L}\) for every vertex \(v\).
If instead \(U(v)\notin\mathscr{L}\), set
\(W'=W\setminus U(v)\).  Then
\(W\setminus W'=U(v)\notin\mathscr{L}\), so the
\(\Sigma\)-Sperner--Shapley labeling condition applies.

Applying the condition to the face
\(\operatorname{supp}_P(v)\), we find \(w\in W'\) and
\(\tau\subseteq\operatorname{supp}_P(v)\) with
\(v\in A^w_\tau\).  Therefore \(w\in U(v)\), contradicting
\(w\in W\setminus U(v)\).

For a simplex \(\theta\in\mathcal T\), define
\[
        D_w(\theta)
        =
        \bigcup_{v\in V(\theta)}D_w(v),
        \qquad
        U(\theta)=\bigcup_{v\in V(\theta)}U(v) =
        \{w:D_w(\theta)\ne\emptyset\}.
\]
Since \(\mathscr{L}\) is upward closed, \(U(\theta)\in\mathscr{L}\).
Set
\[
        \Gamma(\theta)
        =
        \mathcal{K}\left(
        U(\theta);
        \bigl(D_w(\theta)\bigr)_{w\in U(\theta)}
        \right).
\]

By \(d\)-admissibility, \(|\Gamma(\theta)|\) is
\((d-1)\)-connected.  The assignment
\(\theta\mapsto\Gamma(\theta)\) is monotone under inclusion of
simplices.

We now construct the map $\psi$.  Since
\(\dim P=d\), the connectivity of the complexes
\(\Gamma(\theta)\) will allow us to construct a continuous map \( \psi:P\longrightarrow |\mathcal{K}(W;\mathcal F(P))| \)
such that
\(
        \psi(|\theta|)
        \subseteq
        |\Gamma(\theta)|\)
 {for every }\(\theta\in\mathcal T \).
To do this, define \(\psi\) inductively over the skeleta of
\(\mathcal T\).  For a vertex $v$, choose an arbitrary point in
\(|\Gamma(v)|\) for $\psi(v)$.  If \(\psi\) has been defined on the
\((r-1)\)-skeleton and \(\theta\) is an \(r\)-simplex, then
\(\psi(\partial\theta)\subseteq|\Gamma(\theta)|\); since
\(r\le d\) and \(|\Gamma(\theta)|\) is \((d-1)\)-connected, this
boundary map extends over \(\theta\).  

The vertices of $\mathcal{K}(W;\ff(P))$ are the pairs $\alpha=(w,\tau) \in W \times \ff(P)$.  We can write $\psi(x)$ as a convex combination of vertices of $|\mathcal{K}(W;\mathcal F(P))|$, as follows.
\[
        \psi(x)=
        \sum_{\alpha\in W \times \ff(P)}
        \lambda_{\alpha}(x) \alpha.
\]
Using these coefficients, we can define
\begin{align*}
     Y:P & \to P \\
     x & \mapsto
        \sum_{\alpha \in W \times \ff(P)}
        \lambda_{\alpha}(x)y_{\alpha}
\end{align*}

Let us show that the map \(Y\) is face-preserving.  If \(x\in\sigma\in \ff(P)\)
and \(x\in|\theta|\subseteq\sigma\), then every vertex $\alpha$  in the support of \(\psi(x)\) lies in
\(\Gamma(\theta)\).  This means that \(\tau\in D_w(v)\) for some
\(v\in V(\theta)\), so \(\tau\subseteq \operatorname{supp}_{P}(v)\subseteq \sigma\).
Therefore, \(y_{\alpha}\in\sigma\), and so \(Y(x)\in\sigma\).

By \cref{cor:face-preserving}, the map $Y$ is surjective.  Choose
\(x\in P\) with \( Y(x)=p\).
Let \(\theta\) be a simplex containing \(x\).  The support \( \eta=
        \{\alpha:\lambda_{\alpha}(x)>0\}\)
is a face of \(\Gamma(\theta)\), and
\[
        p\in
        \operatorname{conv}
        \{y_{\alpha}:\alpha \in\eta\}.
\]
By the completion condition of the $d$-admissibility, we can extend \(\eta\) to an element
\[
        \beta=\{(w,\tau_w):w\in B\},
        \qquad
        B\in \operatorname{Facets}(\mathcal{K}).
\]
Since $\eta \subseteq \beta$, we still have
\( p\in
        \operatorname{conv}
        \{y_{\alpha}:\alpha \in \beta\}\).

For each \(w\in B\), the label \(\tau_w\) belongs to
\(D_w(\theta)\), so there is a vertex
\(v_w\in V(\theta)\) with
\(
        v_w\in A_{\alpha}\) for \(\alpha = (w,\tau_w) \in \beta\).
\end{proof}

Now we are ready to prove \cref{thm:selection-structure-komiya}

\begin{proof}[Proof of \cref{thm:selection-structure-komiya}]

For a face-refining triangulation $\mathcal{T}$ of $P$, the sets $V(\mathcal{T}) \cap A^w_{\tau}$ form a Sperner--Shapley labeling of $\mathcal{T}$.

Take a sequence $\mathcal{T}_m$ of face-refining triangulations of $P$ such that their mesh tends to $0$.  We can apply \cref{thm:selection-structure-shapley} to each $\mathcal{T}_m$.  Let $z_m$ be a point in the simplex of $\mathcal{T}_m$ from the conclusion of \cref{thm:selection-structure-shapley}.

Passing to a subsequence if necessary, we may assume without loss of generality that the face $B \in \operatorname{Facets}(\mathcal{K})$ and the faces $\tau_w$ for $w\in B$ from \cref{thm:selection-structure-shapley} in each $\mathcal{T}_m$ are constant.  Additionally, by the compactness of $P$ we may also assume that $z_m$ converges to some point $z \in P$.

Since each set $A^w_{\tau}$ is closed and the mesh of the triangulations $\mathcal{T}_m$ is arbitrarily small, we obtain $z \in A^w_{\tau_w}$ for all $w \in B$, as we wanted to show.
\end{proof}

As mentioned in the introduction, \cref{thm:selection-structure-komiya} implies the following theorem by McGinnis

\begin{theorem}[McGinnis 2024 \cite{McGinnis2024-arxiv}]
    Let $P$ be a $d$-dimensional polytope, $p \in P$, and $M$ be a matroid of rank $d+1$ on the ground set $W$.  Let $r_M$ be the rank function of $M$.  For every $w \in W$ and every face $\tau$ of $P$, let $y^w_{\tau}$ be a point of $\tau$ and $A^w_{\tau} \subseteq P$ be a closed set such that for all $W' \subset W$ with $r_M(W \setminus W') \le d$ and every face $\sigma$ of $P$, we have
    \[
    \sigma \subset \bigcup_{\tau \subseteq \sigma} \bigcup_{w \in W'} A^w_{\tau}.
    \]
    Then there exists a basis $B$ of $M$ and faces $\tau_w$ for $w \in B$ of $P$ such that $p \in \conv\{y^w_{\tau_w}: w \in B\}$ and $\bigcap_{w \in B}A^w_{\tau_w}\neq \emptyset$.
\end{theorem}

\begin{proof}
    Consider the selection structure $\Sigma_M = (\Ind(M)^{(d)}, \mathscr{L}_{M,d})$ from \cref{example:matroids}.  Since $M$ has rank $d+1$, then $\Ind(M)=\Ind(M)^{(d)}$.  The condition $r(W \setminus W') \le d$ is equivalent to $W \setminus W' \not\in \mathscr{L}_{M,d}$.  Therefore, the sets $A^w_{\tau}$ form an $\Sigma_M$-Komiya cover of $P$.  We can apply \cref{thm:selection-structure-komiya} to this family.  A containment-maximal face of $\Ind(M)$ is a basis, so we obtain the desired conclusion.
\end{proof}

\section{Selection structure versions of Carath\'eodory's and Helly's theorems.}\label{sec:helly-and-caratheodory}

\subsection{Selection structure Carath\'eodory}

We present two proofs of \cref{thm:selection-structure-colorful-caratheodory}.  The first proof is direct, and does not use the completion condition for selection structure.  The second proof shows that \cref{thm:selection-structure-colorful-caratheodory} is a consequence of \cref{thm:selection-structure-komiya}.

Frick and Zerbib proved a polytopal colorful Komiya KKMS theorem \cite{Frick2019}.  Additionally, they showed how to use their main result to give a new proof of B\'ar\'any's colorful Carath\'eodory theorem, \cref{thm:colorful-caratheodory}.  For our second proof of \cref{thm:selection-structure-colorful-caratheodory}, we adapt the Frick--Zerbib ideas to use \cref{thm:selection-structure-komiya} instead.

\begin{proof}[First proof of \cref{thm:selection-structure-colorful-caratheodory}]
    For each $w \in W$, let $D_w = \{w\}\times X_w$ and $D = \{D_w:w\in W\}$.  Consider the piecewise linear map
    \[
    f: |\mathcal{K}(W;D)|\to \rr^d
    \]
    such that $f(w,x) = x$ for each $w \in W$ and $x \in X_w$, and extended linearly on the faces of $\mathcal{K}(W;D)$.

    If $0 \in f(|\mathcal{K}(W;D)|)$, let $z$ be a point such that $f(z) = 0$ and let $\eta$ be a face that contains $z$.  Then, $\src(\eta)$ is a face of $\mathcal{K}$ satisfying the conclusion.

    Now assume for a contradiction that $0 \not\in f(|\mathcal{K}(W;D)|)$.  Let $\varepsilon= \dist (0,f(|\mathcal{K}(W;D)|))>0$ which is well defined since $f(|\mathcal{K}(W;D)|)$ is compact.

    For each $q \in S^{d-1}$, let $D_w(q) =\{x \in X_w: \langle q,x\rangle < \varepsilon/3\}$ and ${U(q) = \{w \in W: D_w(q) \neq \emptyset\}}$.

    We first show that $U(q) \in \mathscr{L}$ for all $q \in S^{d-1}$.  If this fails for some $q$, let $W'=W\setminus U(q)$.  We apply the condition of the problem to $W \setminus W' = U(q) \not\in\mathscr{L}$ and get
    \[
    0 \in \conv \left(\bigcup_{w \in W'} X_w\right)
    \]
    By the definition of $U(q)$ for every $w \in W'$ and every $x\in X_w$ we have $\langle x,q\rangle \ge \varepsilon/3$, which would separate all such $X_w$ from $0$.  Therefore $U(q) \in \mathscr{L}$.

    Now, for each $q \in S^{d-1}$, choose an open neighborhood $O_q$ of $q$ such that for every $w \in U(q)$, $x \in D_w(q)$ and every $q' \in O_q$ we have  $\langle q',x\rangle < \varepsilon/2$.  We can do this because the sets $D_w(q)$ are finite.  Take a sufficiently fine triangulation $T$ of $S^{d-1}$ such that for each vertex $v \in V(T)$ there is a point $q_v \in S^{d-1}$ such that the closed star $\St(v)$ of $v$ satisfies $\St(v) \subseteq O_{q_v}$.

    For a simplex $\theta \in T$ with vertices $V(\theta)$, we define the sets
    \[
    U_{\theta}=\bigcup_{v \in V(\theta)}U_{q_v}, \qquad D^{\theta}_w = \bigcup_{ \substack{{v\in V(\theta)}\\ {w \in U_{q_v}}}} D_w(q_v).
    \]
    Since $U_{q_v} \in \mathscr{L}$ and $\mathscr{L}$ is upward closed, we have $U_{\theta}\in \mathscr{L}$.  The set of fibers we use will be $D^{\theta}=\{D^{\theta}_w : w \in W\}$.  Finally, let $\Gamma(\theta) = \mathcal{K}(U_{\theta}; D^{\theta})$.

    Since $\Sigma$ is $d$-admissible, we know that $\Gamma(\theta)$ is $(d-1)$-connected.  Moreover, if $\theta' \subset \theta$ we also have $\Gamma(\theta') \subset \Gamma(\theta)$.

    We now construct a continuous map
    \[
    s: S^{d-1}\to |\mathcal{K}(W;D)|
    \]
    such that for all simplices $\theta$ of $T$ we have $s(|\theta|) \subset |\Gamma(\theta)|$.

    This is done inductively over the skeleta of $T$.  On a vertex $v$, choose an arbitrary point in $|\Gamma(v)|$.  If $s$ has been constructed on the $(r-1)$-skeleton and $\theta$ is an $r$-dimensional simplex, then for all faces $\theta' \subset \theta$ we have $s(|\theta'|) \subset |\Gamma (\theta')| \subset |\Gamma(\theta)|$.  As $\Gamma(\theta)$ is $(d-1)$-connected, the map on $\partial \theta$ extends to $\theta$.

    Now define
    \begin{align*}
    h: |\mathcal{K}(W;D)| & \to S^{d-1} \\
    h(y) & = \frac{f(y)}{\|f(y)\|}.
    \end{align*}
    This is well defined since the image of $f$ has no zeros.  Our final map is
    \begin{align*}
        \phi : S^{d-1} & \to S^{d-1} \\
        \phi = h \circ s.
    \end{align*}

    We claim that $\phi$ has non-zero degree.  Let $q \in S^{d-1}$ and let $\theta \in T$ be a simplex with $q \in |\theta|$.  Since $s(q) \in |\Gamma(\theta)|$ every vertex of the minimal simplex containing $s(q)$ is a point $x \in D_w(q_v)$ for some $v \in V(\theta)$.  However, $q \in |\theta| \subseteq \St (v) \subseteq O_{q_v}$.  This implies that $\langle q,x \rangle < \varepsilon/2$.  Since this inequality is preserved over convex combination, we have $\langle q, f(s(q)) \rangle < \varepsilon/2$.

    We know that $\|f(s(q))\|\ge \varepsilon $, which gives us
    \[
    \langle q, \phi(q)\rangle = \left\langle q, \frac{f(s(q))}{\|f(s(q))\|} \right\rangle < \frac{1}{\varepsilon} \cdot \frac{\varepsilon}{2}< \frac{1}{2}
    \]
    This implies that $\phi(q) \neq q$.  We can take the standard straightline homotopy between $\phi(q)$ and $-q$ to show that $\phi(q)$ is homotopic to the antipodal map.  More precisely, we take
    \begin{align*}
        H: [0,1] \times S^{d-1} & \to S^{d-1} \\
        (\lambda,q) & \mapsto \frac{(1-\lambda)\phi(q) - \lambda q}{\|(1-\lambda)\phi(q) - \lambda q\|}.
    \end{align*}
    Since the antipodal map has degree $(-1)^{d}\neq 0$, the same holds for $\phi$.

    On the other hand, since $\mathcal{K}(W;D)$ is $(d-1)$-connected, the map $s: S^{d-1}\to |\mathcal{K}(W;D)|$ extends to a map $\overline{s}: B^d \to |\mathcal{K}(W;D)|$, where $B^d \subset \rr^d$ is the closed unit ball.  The map $h \circ \overline{s}: B^d \to S^{d-1}$ extends $\phi$, which means that the degree of $\phi$ must be zero.  This is the contradiction we wanted.
\end{proof}

Now we present the proof that uses the selection structure Komiya theorem.

\begin{proof}[Second proof of \cref{thm:selection-structure-colorful-caratheodory}]
We derive the result from \cref{thm:selection-structure-komiya}.  For a face
\(\tau\) of a polytope \(P\) containing the origin, write \(
        \operatorname{cone}(\tau)
        =
        \{\lambda x:\lambda\ge 0,\ x\in \tau\}
\) for the cone over \(\tau\).

First, suppose that \(0\in X_w\) for some \(w\in W\).  Applying the completion condition to the singleton
face supported at \(w\), we can extend \(w\) to some \(B\in \operatorname{Facets}(\mathcal{K})\).
Choosing \(x_w=0\) and arbitrary points \(x_u\in X_u\) for
\(u\in B\setminus\{w\}\) gives the conclusion.  We assume from
now on that \(0\notin \bigcup_{w\in W}X_w\).

Deleting all but one point on each ray from the origin in each \(X_w\) does
not affect the hypothesis nor the conclusion.

Choose a \(d\)-dimensional polytope \(P\subset \rr^d\) with
\(0\in\relint(P)\), sufficiently fine around the origin, so that the following
two conditions hold.

\begin{enumerate}
\item If two nonzero points lie in the cone over the same proper face of
\(P\), then their scalar product is positive.

\item For every proper face \(\tau\) of \(P\) and every \(w\in W\), the set
\(\operatorname{cone}(\tau)\cap X_w\) has at most one point.
\end{enumerate}

For every \(w\in W\) and every face \(\tau\) of \(P\), we now define closed
sets \(A^w_\tau\subseteq P\) and points \(y^w_\tau\in\tau\).

If \(\tau=P\), choose \(y^w_P\in P\) arbitrarily and set \( A^w_P=\emptyset\).
Now let \(\tau\) be a proper face of $P$.  If
\(\operatorname{cone}(\tau)\cap X_w=\{x^w_\tau\}\), let \(y^w_\tau\) be the intersection of
\(\tau\) with the ray from the origin through \(x^w_\tau\), and set
\[
        A^w_\tau
        =
        \{q\in P:\langle q,x^w_\tau\rangle\ge 0\}.
\]
If \(\operatorname{cone}(\tau)\cap X_w=\emptyset\), choose \(y^w_\tau\in\tau\) arbitrarily and
set \( A^w_\tau=\tau\).

We claim that the family \( \bigl(A^w_\tau
        \mid
        w\in W,\ \tau \mbox{ a face of }P\bigr) \)
is a \(\Sigma\)-Komiya cover of \(P\).  Let
\(W'\subseteq W\) be such that \(W\setminus W'\notin\mathscr{L}\).  By the
hypothesis, \(0\in\conv(\bigcup_{w\in W'}X_w)\).

First, let \(\sigma\) be a proper face of \(P\), and let \(q\in\sigma\).
Choose any \(w\in W'\).  If \(\operatorname{cone}(\sigma)\cap X_w=\emptyset\), then
\(A^w_\sigma=\sigma\), so \(q\in A^w_\sigma\).  If
\(\operatorname{cone}(\sigma)\cap X_w=\{x^w_\sigma\}\), then \(x^w_\sigma\) lies in the cone
over \(\sigma\).  By the choice of \(P\),
\( \langle q,x^w_\sigma\rangle\ge 0\),
and therefore \(q\in A^w_\sigma\).  This means that
\[
        \sigma
        \subseteq
        \bigcup_{\tau\subseteq\sigma}
        \ \bigcup_{w\in W'} A^w_\tau
\]
for every proper face \(\sigma\).

It remains to check the face \(\sigma=P\).   Let \(q\in P\).  Since \( 0\in
        \conv
        \Bigl(\bigcup_{w\in W'}X_w\Bigr)\),
there exist \(w\in W'\) and \(x\in X_w\) such that
\( \langle q,x\rangle\ge 0\).
Let \(\tau\) be the proper face of \(P\) whose cone contains \(x\).  By
the construction of \(P\), \(x=x^w_\tau\).  Therefore \(q\in A^w_\tau\).
This proves the \(\Sigma\)-Komiya covering condition.

Apply \cref{thm:selection-structure-komiya} with the point \(p=0\in P\).
There exist \(B\in \operatorname{Facets}(\mathcal{K})\), faces \(\tau_w\) of \(P\) for \(w\in B\),
and a point
\[
        z\in \bigcap_{w\in B} A^w_{\tau_w}
\]
such that \( 0\in
        \conv\{y^w_{\tau_w}:w\in B\}\).
Since \(A^w_P=\emptyset\), all faces \(\tau_w\) are proper faces.

Choose the coefficients \(\lambda_w\ge 0\), \(w\in B\), of a convex combination that shows  \( 0\in
        \conv\{y^w_{\tau_w}:w\in B\}\).
\[
        \sum_{w\in B}\lambda_w=1,
        \qquad
        \sum_{w\in B}\lambda_w y^w_{\tau_w}=0.
\]
Taking scalar product with \(z\), we get
\(
        0
        =
        \sum_{w\in B}
        \lambda_w\langle z,y^w_{\tau_w}\rangle
\).  If
\({\operatorname{cone}(\tau_w)\cap X_w=\{x^w_{\tau_w}\}}\), then \(y^w_{\tau_w}\) is a positive
multiple of \(x^w_{\tau_w}\), and \(z\in A^w_{\tau_w}\) gives \(\langle z,y^w_{\tau_w}\rangle\ge 0\).  Therefore, each term in the sum \(\sum_{w\in B}
        \lambda_w\langle z,y^w_{\tau_w}\rangle\) is nonnegative, and must therefore be zero.
        
If \(\operatorname{cone}(\tau_w)\cap X_w=\emptyset\), then \(A^w_{\tau_w}=\tau_w\), so both
\(z\) and \(y^w_{\tau_w}\) lie in the same proper face \(\tau_w\).  By the
choice of \(P\),
\( \langle z,y^w_{\tau_w}\rangle>0
\), so $\lambda_w =0$.

Therefore, for every \(w\in B\) with \(\lambda_w>0\), there is a point
\(x_w\in X_w\cap \operatorname{cone}(\tau_w)\), and
\( y^w_{\tau_w}=a_w x_w\)
for some scalar \(a_w>0\).

For \(w\in B\) with \(\lambda_w=0\), choose \(x_w\in X_w\) arbitrarily.
Then
\[
        0
        =
        \sum_{\lambda_w>0}
        \lambda_w y^w_{\tau_w}
        =
        \sum_{\lambda_w>0}
        \lambda_w a_w x_w .
\]
After normalizing the positive coefficients \(\lambda_w a_w\), this says
\( 0\in
        \conv\{x_w:\lambda_w>0\}\).
In particular
\(0\in
        \conv\{x_w:w\in B\}\),
as we wanted to prove.
\end{proof}

\subsection{Selection structure Helly}\label{sec:helly}

In this subsection, we prove \cref{thm:selection-structure-helly}.  We prove two generalizations of \cref{thm:selection-structure-helly}: one for $d$-collaspisble simplicial complexes, and the second for $d$-Leray selection complexes.  The version for $d$-Leray selection complexes is more general, but we keep the proof of both since the direct proof for $d$-collapsible simplicial complexes is much simpler.  Since both imply the Kalai--Meshulam topological Helly theorem when it is applied to families of convex sets, we believe they are worth presenting.

Let $X$ be a simplicial complex and $\sigma$ be a face such that
\begin{itemize}
    \item $\dim \sigma \le d-1 $ and
    \item $\sigma$ is contained in a unique containment-maximal face $\eta$.
\end{itemize}

If we remove all faces that contain $\sigma$, we obtain the simplicial complex $X' = X \setminus \{\tau : \sigma \subseteq \tau \}$.  We call this process an elementary $d$-collapse, and denote it $X \searrow X'$.

The simplicial complex $X$ is $d$-collapsible if there exists a finite sequence of $d$-collapses that starts with $X$ and ends with $\emptyset$.  Namely,
\[
X =X_0 \searrow X_1 \searrow X_2 \searrow \dots \searrow X_m = \emptyset.
\]

If $\mathcal{F}$ is a finite family of convex sets in $\rr^d$, then its nerve complex $X = N(\mathcal{F})$ is $d$-collapsible \cites{Wegner1975, Tancer2013}.

We prove our selection structure Helly for $d$-collapsible simplicial complexes.

\begin{theorem}\label{thm:selection-structure-d-collapsible}
    Let $W$ be a finite set and $X$ be $d$-collapsible simplicial complex with vertex set $W$.  Let $\Sigma=(\mathcal{K}, \mathscr{L})$ be a $d$-admissible selection structure on $W$.  Suppose that for every $W' \subseteq W$ such that $W\setminus W' \not\in \mathscr{L}$ we have $W' \not \in X$.  Then, there exists $I \in \mathcal K$ such that $ I \not\in X$.
\end{theorem}

\begin{proof}[Proof of \cref{thm:selection-structure-helly}]
Let $X = N(\mathcal{F})$ be the nerve of the family $\ff$.  Since $X$ is $d$-collapsible, we apply \cref{thm:selection-structure-d-collapsible}, giving us the conclusion we needed.
\end{proof}

\begin{proof}[Proof of \cref{thm:selection-structure-d-collapsible}]
    The idea behind the proof is as follows.  Assume that the conclusion is false.  Then, we will use a sequence of $d$-collapses of $X$ to find a small maximal face $\sigma$ of $\mathcal{K}$.  Then, we will use this face to contradict the connectivity condition of the selection structure $\Sigma$.  We will do this by taking a parallel expansion (i.e., the sets $D_w$ for $w \in W$) in which the vertices of $\sigma$ are given two elements and the rest are given one element.  We will show that this forces non-trivial homology of dimension at most $d-1$.
    
    Now we start the proof.  Let $\mathcal{B} = \operatorname{Facets}(\mathcal{K})$.  Assume, for a contradiction, that $\mathcal{K} \subseteq X$.

    First, observe that if $\tau \in X$, then we must have $ W \setminus \tau \in \mathscr{L}$.  Otherwise, we would be able to apply the condition to $W' = \tau$ and reach a contradiction.  
    
    Now we use the $d$-collapsibility of $X$.  Choose a sequence of elementary $d$-collapses

    \[
X =X_0 \searrow X_1 \searrow X_2 \searrow \dots \searrow X_m = \emptyset.
\]
At each step we obtain $X_{i+1}$ by removing the set of faces
\[
[\sigma_i, \eta_i] := \{\tau \in X_i: \sigma_i \subseteq \tau \subseteq \eta_i\},
\]
where $\sigma_i$ is a face of dimension at most $d-1$ and $\eta_i$ is the unique containment-maximal face of $X_i$ that contains $\sigma_i$.

Let $i$ be the first index such that $[\sigma_i, \eta_i]$ contains a face of $\mathcal{K}$.  This implies that $\sigma_i \in \mathcal{K}$.  To simplify the notation in the rest of the proof, let $\sigma = \sigma_i, \eta = \eta_i$.

We have $\eta \setminus \sigma \in X$.  Therefore, $U = W \setminus (\eta \setminus \sigma) \in \mathscr{L}$ and $\sigma \subseteq U$.

If we apply the completion axiom to $U$, again with singleton fibers $D_w = \{w\}$, we find some $B_0 \in \mathcal{B}$ such that $\sigma \subseteq B_0 \subseteq U$.  Since $B_0 \in \mathcal{B}\subseteq \mathcal{K}$ and $i$ was the first index in which we remove a face of $\mathcal{K}$, we have $B_0 \in X_i$.

Moreover, since $\eta$ was the unique maximal face of $X_i$ that contained $\sigma$, we have $B_0 \subseteq \eta$.  However, $B_0 \subseteq W \setminus (\eta \setminus \sigma)$, which means that $B_0 \cap (\eta \setminus \sigma) = \emptyset$.  This means that $B_0$ is contained in $\eta$ but lacks any of the vertices of $\eta \setminus \sigma$, so $B_0 = \sigma$.  This translates simply to $\sigma \in \operatorname{Facets}(\mathcal{K})$.

We can rule out $\sigma = \emptyset$ since $\mathcal{K}$ is at least $(-1)$-connected and therefore non-empty, so its facets have at least one element.

Recall that $k:=\dim \sigma \le d-1$.  We now use $\sigma$ to contradict the connectivity condition of the selection structure $\Sigma$.

Choose sets $D_w$ for $w \in W$ such that
\[
|D_w| = \begin{cases}
    2 & \mbox{ if } w \in \sigma \\
    1 & \mbox{otherwise}.
\end{cases}
\]
In the complex $\mathcal{K}(W;D)$, the vertices of $\sigma$ induce the $(k+1)$-fold join of two points, which is homeomorphic to a $k$-dimensional sphere $S^{k}$.  Moreover, since $\sigma$ was a maximal face, no top-dimensional simplex of this copy of $S^k$ is properly contained in another simplex of $\mathcal{K}(W;D)$.  In particular, this copy of $S^k$ induces a non-zero element of the $k$-th reduced homology group of $\mathcal{K}(W;D)$, contradicting the fact that $\mathcal{K}(W;D)$ should be $(d-1)$-connected.  This is the contradiction we were looking for.
\end{proof}

As mentioned in the introduction, \cref{thm:selection-structure-helly} implies the Kalai--Meshulam matroidal Helly theorem.  Kalai and Meshulam proved two versions of this result.  One concerning convex sets, \cite{Kalai2005}*{Thm. 1.3}, which we explain here.  They also proved a strictly stronger second theorem, \cite{Kalai2005}*{Thm. 1.6}, for $d$-Leray simplicial complexes.  We explain how \cref{thm:selection-structure-d-collapsible} implies \cite{Kalai2005}*{Thm. 1.6} below.  Analogous arguments show that that our generalization to $d$-Leray complexes, \cref{thm:selection-structure-d-leray}, implies \cite{Kalai2005}*{Thm. 1.3}.

\begin{theorem}[Kalai, Meshulam 2005 \cite{Kalai2005}*{Thm. 1.3}]\label{thm:kalai}
    Let $d$ be a positive integer, $M$ be a matroid on a finite set $W$ of vertices, and $r_M$ be the rank function of $M$.  Let $X$ be the nerve complex of a finite family of convex sets in $\rr^d$.  Suppose that $\Ind(M) \subseteq X$.  Then, there exists a simplex $W' \in X$ such that $r_M(W \setminus W') \le d$.
\end{theorem}

\begin{proof}[Proof that \cref{thm:selection-structure-helly} implies \cref{thm:kalai}]

If the rank of $M$ is at most $d$, we can take $W'$ to be basis of $M$.  Therefore we can assume that the rank of $M$ is at least $d+1$.

Suppose for a contradiction that the result fails.  Let $\Sigma_M = (\Ind(M)^{(d)}, \mathscr{L}_{M,d})$ from \cref{example:matroids}.  Since the conclusion of the theorem fails, whenever $W \setminus W' \not\in \mathscr{L}_{M,d}$, we have $r_M(W \setminus W') \le d$.  Therefore, $W' \not\in X$.  This implies we can use \cref{thm:selection-structure-d-collapsible}.  The conclusion contradicts the assumption $\Ind(M)^{(d)} \subseteq \Ind(M) \subseteq X$, proving the theorem.
    
\end{proof}

The next result is when $X$ is a $d$-Leray simplicial complex.  Recall that a simplicial complex $X$ with vertex set $W$ is $d$-Leray (over a field $\mathds{F}$ if for all $U \subseteq W$, the reduced homology groups $\tilde H_i(X[U], \mathds{F})$ are zero for every $i \ge d$.  In other words, the homology of every induced subcomplex vanishes in dimension at least $d$.  We can use $\mathds{F}=\mathds{Q}$ in the arguments that follow.

In particular, if $\mathcal{F}$ is a finite family of non-empty open sets in $\rr^d$ such that any intersection of them is either empty or contractible, its nerve complex is $d$-Leray.  This is the kind of families on which Kalai and Meshulam apply their results.  Additionally, $d$-collapsible complexes are $d$-Leray, which is why \cref{thm:selection-structure-d-leray} generalizes \cref{thm:selection-structure-d-collapsible}.  We now state the selection-structure Helly theorem for $d$-Leray simplicial complexes.   

\begin{theorem}\label{thm:selection-structure-d-leray}
    Let $W$ be a finite set and $X$ be a $d$-Leray simplicial complex with vertex set $W$.  Let $\Sigma=(\mathcal{K}, \mathscr{L})$ be a $d$-admissible selection structure on $W$.  Suppose that for every $W' \subseteq W$ such that $W \setminus W' \not\in \mathscr{L}$ we have $ W' \not \in X$.  Then, there exists $I \in \mathcal K$ such that $ I \not\in X$.
\end{theorem}

The proof is based on a very recent result by Blagojevi\'c, Karasev, and Zsigri \cite{blagojevic2026cone}.  In their work, the Blagojevi\'c et al present extensions of Carath\'eodory's theorem that generalize resutls by Holmsen, Kalai, and Meshulam.  In particular, the result relevant for this proof, \cite{blagojevic2026cone}*{Thm. 3.2}, extends earlier methods by Meshulam that show the existence of rainbow simplices in simplicial complexes where the vertices are assigned colors.  We state the relevant result here for completeness, but highlight that the statement by Blagojevi\'c, Karasev, and Zsigri is more general (one of the conditions can be modified).  We also apply directly a change of variables to make the result below match the parameters of \cref{thm:selection-structure-d-leray}.

\begin{theorem}[Blagojevi\'c, Karasev, Zsigri 2026 \cite{blagojevic2026cone}*{Thm. 3.2}]\label{thm:Blagojevic-new}
    Let $d$ be a positive integer, $W$ be a non-empty set, $X$ be a $d$-Leray complex with ground set $W$.  We know that for all $\sigma \in X$, we have that $\mathcal{K}[W \setminus \sigma]$ is homologically $(d-1)$-connected.  Then, $\mathcal{K} \not\subseteq X$.
\end{theorem}

\begin{proof}[Proof of \cref{thm:selection-structure-d-leray}]
    We may assume without loss of generality that $W$ is the set of vertices of $\mathcal{K}$.  If $\mathcal{K}$ does not use all of $W$, we can consider $V$ its set of vertices, and denote $X' = X[V]$, which is also $d$-Leray, and the arguments below follow.
    
    Since $\mathscr{L}$ is not empty and it is upwards-closed, $W \in \mathscr{L}$. In particular, if we take singleton fibers $D_w = \{w\}$, we have that $\mathcal{K}(W;D) = \mathcal{K}$ is $(d-1)$-connected, and in particular must not be empty.

    We have to show the homological connectivity condition of \cite{blagojevic2026cone}*{Thm. 3.2}.  Let $\sigma \in X$ be a face.  We claim that $W \setminus \sigma \in \mathscr{L}$.  Otherwise, we can take $W' = \sigma$ in the hypothesis and obtain that $\sigma \not\in X$.

    Now consider again singleton fibers $D_w = \{w\}$, we have that $\mathcal{K}(W\setminus \sigma; D) = \mathcal{K}[W \setminus \sigma]$ is $(d-1)$-connected for every $\sigma \in X$.  Therefore, by \cref{thm:Blagojevic-new}, $\mathcal{K} \not\subseteq X$.  This implies the existence of some face $I \in \mathcal{K}$ such that $I \not\in X$, as we wanted to prove.
\end{proof}

\section{Selection structure Goodman--Pollack transversals}
\label{sec:selection-structure-goodman-pollack}

We isolate the two features of Sadovek’s proof needed to extend to matroids.  The proof uses a Stiefel-manifold
Borsuk--Ulam-type theorem, together with an auxiliary space built from the
dependency condition.  The information from the matroid enters through two features: the high
connectivity of the relevant parallel extensions, and the fact that faces extend
to bases.  These are precisely the two axioms in the definition of an
admissible selection structure.

We now introduce the definitions and notation needed to state \cref{thm:selection-structure-sadovek}.  As Sadovek's theorem holds for affine transversals in real spaces and in complex spaces, we prove the selection structure version on this level of generality.

Let \(\FF\in\{\rr,\cc \}\).  An \(\FF\)-affine \(k\)-flat in
\(\FF^d\) is an \(\FF\)-\(k\)-transversal to a family of convex sets if it
intersects every member of the family.  One reason that complex transversals are combinatorially interesting is that even though $\cc^d \cong \rr^{2d}$, not every real affine $2k$-dimensional space of $\rr^{2d}$ corresponds to an affine complex $k$-dimensional space of $\cc^d$.  The following definition is one of the central points that McGinnis and Sadovek used to characterize the existence of transversals to families of convex sets.

\begin{definition}
\label{def:selection-structure-models-dependencies}
Let \(0\le r\le k<d\), let \(m=d-k\), and let \( \ff=\{F_w:w\in W\} \) be a finite family of compact convex sets in \(\FF^d\).  Let
\(\Sigma=(\mathcal{K},\mathscr{L})\) be a selection structure on
\(W\) and let $\mathcal{B} = \operatorname{Facets}(\mathcal{K})$.  A map
\[
        \varphi:W\longrightarrow \FF^r
\]
models \((d-k,\Sigma,\FF)\)-dependencies of \(\ff\) if the following
condition holds.

For every \(B\in\mathcal B\) and every nontrivial system of \(m=d-k\)
affine \(\FF\)-dependencies
\[
        \sum_{w\in B} a_{j,w}=0,
        \qquad
        \sum_{w\in B} a_{j,w}\varphi(w)=0
        \qquad
        \text{for }j=1,\dots,m,
\]
where not all \(a_{j,w}\) are zero, there exist real numbers
\(\lambda_w\ge 0\) and points \(q_w\in F_w\), \(w\in B\), such that
\[
        \sum_{w\in B} \lambda_w a_{j,w}=0,
        \qquad
        \sum_{w\in B} \lambda_w a_{j,w}q_w=0
        \qquad
        \text{for }j=1,\dots,m,
\]
and the new system is still nontrivial, i.e. not all
\(\lambda_w a_{j,w}\) are zero.
\end{definition}

To prove \cref{thm:selection-structure-sadovek}, we use the following Borsuk--Ulam-type theorem for Stiefel manifolds.  Let \( V_m(\FF^{d+1}) \)
be the Stiefel manifold of orthonormal \(m\)-frames
\(v=(v_1,\dots,v_m)\) in \(\FF^{d+1}\).  We use the standard action of
\((\ZZ_2)^m\) on \(V_m(\FF^{d+1})\). To do this, we consider $\ZZ_2 = \{-1,1\}$ with multiplication, and for $(v_1,\dots,v_m) \in V_m(\FF^{d+1})$ and $(\lambda_1,\dots, \lambda_{m}) \in (\ZZ_2)^m$, we take
\[
(\lambda_1,\dots, \lambda_{m}) (v_1,\dots,v_m) =(\lambda_1 v_1,\dots, \lambda_{m}v_{m}).
\]

We use the following theorem by Sadovek.  This is a generalization of earlier Borsuk--Ulam theorems for Stiefel manifolds \cites{Chan2020, Sadovek2026}.

\begin{theorem}[Sadovek \cite{sadovek2026colorful}*{Thm. 3.1}]
\label{thm:stiefel-obstruction}
Let \(0\le r\le k<d\), let \(\FF\in\{\rr,\cc\}\) and \( N=\dim_{\rr}(\FF)\,(d-k)(r+1)\).
Endow the Stiefel manifold \( V_{d-k}(\FF^{d+1}) \) with a \((\zz_2)^{d-k}\)-CW structure compatible with the action described above.  Then, there is no continuous \((\zz_2)^{d-k}\)-equivariant map
\[
        \Phi:\operatorname{sk}_{N} V_{d-k}(\FF^{d+1})
        \longrightarrow
        S\bigl((\FF^{r+1})^{d-k}\bigr),
\]
where the \(j\)-th factor of \((\zz_2)^{d-k}\) acts antipodally on the
\(j\)-th copy of \(\FF^{r+1}\) and trivially on the other copies.
\end{theorem}

\begin{theorem}[Selection structure Goodman--Pollack theorem]
\label{thm:selection-structure-sadovek}
Let \(0\le r\le k<d\), let \(\FF\in\{\rr,\cc\}\), and put \( N=\dim_{\rr}(\FF)\,(d-k)(r+1)\).
Let \( \ff=\{F_w:w\in W\} \)
be a finite family of compact convex sets in \(\FF^d\), and let
\(\Sigma=(\mathcal{K},\mathscr{L}) \)
be an \(N\)-admissible selection structure on \(W\).  Suppose that there exists
a map
\[
        \varphi:W\longrightarrow \FF^r
\]
that models \((d-k,\Sigma,\FF)\)-dependencies of \(\ff\).  Then there
exists \(W'\subseteq W\) such that the family
\( \{F_w:w\in W'\} \)
has an \(\FF\)-\(k\)-transversal and
\( W\setminus W'\notin\mathscr{L}\).
\end{theorem}

Before proving \cref{thm:selection-structure-sadovek}, let us show that it implies \cref{thm:selection-structure-helly}.

\begin{proof}[Second proof of \cref{thm:selection-structure-helly}]
    Since $W$ is finite, we can assume without loss of generality that all convex sets in the family $\{F_w: w \in W\}$ are compact.  Let $\mathcal{B} = \operatorname{Facets}(\mathcal{K})$.  We prove the contrapositive, and assume that for all $B \in \mathcal{B}$ we have $\bigcap_{w\in B}F_w \neq \emptyset$.  We set $k=r=0$.  Then, the constant map $\varphi: W \to \rr^{0}=\{0\}$ models $(d,\Sigma,\rr)$-dependencies of $\ff$.

    To see this, for every $B \in \mathcal{B}$ and any set of coefficients $a_{j,w}$ such that $\sum_{w \in B}a_{j,w}=0$, it is enough to take some $q_B \in \bigcap_{w \in B} F_w$ and select the points $q_w = q_B \in F_w$ for all $w \in B$.

    We can apply \cref{thm:selection-structure-sadovek} and conclude that there exists some $W' \subseteq W$ such that $\bigcap_{w \in W'}F_w \neq \emptyset$ and $W \setminus W' \not\in \mathscr{L}$, as we wanted to prove.
\end{proof}

\begin{proof}[Proof of \cref{thm:selection-structure-sadovek}]
The idea behind the proof is to assume that the conclusion of \cref{thm:selection-structure-sadovek} fails, and use this to construct a map ${\Phi:\operatorname{sk}_{N} V_{d-k}(\FF^{d+1}) \rightarrow S\bigl((\FF^{r+1})^m\bigr)}$ that contradicts \cref{thm:stiefel-obstruction} for some sufficiently fine triangulation of $V_{d-k}(\FF^{d+1})$.  We will use the $N$-admissibility of $\Sigma$ to construct $\Phi$ inductively on the skeleta of $V_{d-k}(\FF^{d+1})$.  We define \(m=d-k\).

We argue by contradiction.  Assume that the conclusion is false.  For
every \(T\subseteq W\),

\begin{equation}\label{eq:1-implication}
    T\notin\mathscr{L}
        \quad\Longrightarrow\quad
        \{F_w:w\in W\setminus T\}
        \text{ has no \(\FF\)-\(k\)-transversal.}
\end{equation}

For \(q\in\FF^d\), write \(\widehat q=(q,1)\in\FF^{d+1}\), and denote \( \widehat F_w=\{(q,1):q\in F_w\}\subseteq \FF^{d+1}\).

For \(v\in V_m(\FF^{d+1})\), define
\[
        L(v)=
        \left(\operatorname{span}_{\FF}\{v_1,\dots,v_m\}\right)^{\perp_{\FF}}.
\]
Then \(L(v)\) has \(\FF\)-dimension \(k+1\).  Let
\(U(v)
        =
        \{w\in W:\widehat F_w\cap L(v)=\emptyset\}\).
If \(U(v)\notin\mathscr{L}\), then \(W'=W\setminus U(v)\) satisfies
\(W\setminus W'\notin\mathscr{L}\).  Moreover, for every \(w\in W'\), the set
\(\widehat F_w\) intersects \(L(v)\).  Therefore the affine slice
\[
        L(v)\cap (\FF^d\times\{1\})
\]
is an \(\FF\)-affine \(k\)-flat meeting every \(F_w\) with \(w\in W'\)
(unless \(W'=\emptyset\), in which case any \(k\)-flat works).  This
contradicts \eqref{eq:1-implication}.  Therefore

\begin{equation}\label{eq:2-S(v)-in-S}
    U(v)\in\mathscr{L}
        \qquad\text{for every }v\in V_m(\FF^{d+1}).
\end{equation}

Let
\(Z_m(\FF)=\{z=(z_1,\dots,z_m)\in\FF^m:\|z\|=1\}\).
The action of \((\ZZ_2)^m\) on $\FF^m$ induces an action on \(Z_m(\FF)\).  For a face \(I\in\mathcal{K}\), let \(X_I\subseteq Z_m(\FF)^I\)
be the set of all tuples \((z_w)_{w\in I}\) with the following property:

\begin{quote}
If \(b_w\ge 0\) and \(q_w\in F_w\), \(w\in I\), satisfy
\[
        \sum_{w\in I} b_w z_{w,j}=0,
        \qquad
        \sum_{w\in I} b_w z_{w,j}q_w=0
        \qquad
        \text{for }j=1,\dots,m,
\]
then \(b_w=0\) for every \(w\in I\).
\end{quote}

Define \(X\) as the following \((\ZZ_2)^m\)-invariant subspace of the join
\(\bigast_{w\in W} Z_m(\FF)\):
\[
        X=
        \left\{
        \sum_{w\in I} t_w z_w:
        I\in\mathcal{K},\quad
        t_w>0,\quad
        \sum_{w\in I}t_w=1,\quad
        (z_w)_{w\in I}\in X_I
        \right\}.
\]
This is the auxiliary space used in Sadovek's proof, written here in join
notation instead of homotopy-colimit notation.

We use the Hermitian inner product on \(\FF^{d+1}\) that is linear in
the second coordinate.  For \(v\in V_m(\FF^{d+1})\) and \(w\in U(v)\), define
\[
        D_w(v)
        =
        \left\{
        z\in Z_m(\FF):
        \operatorname{Re}
        \left\langle
        \overline{z_1}v_1+\cdots+\overline{z_m}v_m,\widehat q
        \right\rangle_{\FF}>0
        \text{ for all }\widehat q\in \widehat F_w
        \right\}.
\]
We claim that the set \(D_w(v)\) is not empty.  The condition \(w\in U(v)\) means that
\(\widehat F_w\) is disjoint from \(L(v)\), so the orthogonal projection of
\(\widehat F_w\) to \(\operatorname{span}_{\FF}\{v_1,\dots,v_m\}\) is a
compact convex set not containing the origin.  Strict separation gives the
desired vector \(z\), after normalization.

Additionally, if \(I\in\mathcal{K}[U(v)]\) and
\(z_w\in D_w(v)\) for all \(w\in I\), then \((z_w)_{w\in I}\in X_I\).
This holds because if \(b_w\ge 0\) and \(q_w\in F_w\) satisfy
\[
        \sum_{w\in I} b_w z_{w,j}=0,
        \qquad
        \sum_{w\in I} b_w z_{w,j}q_w=0
        \qquad
        \text{for }j=1,\dots,m,
\]
we can rewrite this as
\( \sum_{w\in I} b_w z_{w,j}\widehat q_w=0\) in \(\FF^{d+1}\)
         for \(j=1,\dots,m\).
Taking inner product \(\langle v_j, \cdot \rangle_{\FF}\), summing over \(j\), and taking real
parts gives
\[
        0
        =
        \sum_{w\in I}
        b_w
        \operatorname{Re}
        \left\langle
        {\overline{z_{w,1}}}v_1+\cdots+\overline{z_{w,m}}v_m,\widehat q_w
        \right\rangle_{\FF}.
\]
Each summand inside the real part is strictly positive by the definition of
\(D_w(v)\).  Therefore every \(b_w\) is zero and the tuple satisfies the
defining condition for \(X_I\).

We now construct an equivariant map from the \(N\)-skeleton of the Stiefel
manifold into \(X\).  For every \(v\in V_m(\FF^{d+1})\), choose points
\(z_w^v\in D_w(v)\) for \(w\in U(v)\).  Since all inequalities defining
\(D_w(v)\) are strict and the sets \(\widehat F_w\) are compact, there is an
open neighborhood \(O_v\) of \(v\) such that

\begin{equation}\label{eq:3-in-D-w}
    z_w^v\in D_w(u)
        \qquad
        \text{for every }u\in O_v
        \text{ and every }w\in U(v).
\end{equation}
Choose these neighborhoods and points equivariantly with respect to the
\((\ZZ_2)^m\)-action.

Take a sufficiently fine \((\ZZ_2)^m\)-equivariant triangulation of
\(V_m(\FF^{d+1})\) such that, for every simplex \(\tau\), there is a point
\(v_\tau\) with \( \operatorname{st}(\tau)\subseteq O_{v_\tau}\).
Choose the points \(v_\tau\) equivariantly over orbits of simplices.

For a simplex \(\theta\), define
\[
        U_\theta=
        \bigcup_{\tau\subseteq\theta} U(v_\tau).
\]
Given \eqref{eq:2-S(v)-in-S} and the fact that $\mathscr{L}$ is upwards closed, we have \(U_\theta\in\mathscr{L}\).  For
\(w\in U_\theta\), let
\( D_w(\theta)
        =
        \{z_w^{v_\tau}:
        \tau\subseteq\theta,\ w\in U(v_\tau)\}\).
This is a non-empty finite set.  Set
\( \Gamma(\theta)
        =
        \mathcal{K}
        \left(
        U_\theta;
        \bigl(D_w(\theta)\bigr)_{w\in U_\theta}
        \right)\).
The assignment \(\theta\mapsto \Gamma(\theta)\) is monotone under inclusion
of simplices.  Moreover, by \(N\)-admissibility,
\begin{equation}\label{eq:4-connected}
            |\Gamma(\theta)|
        \text{ is }(N-1)\text{-connected}.
\end{equation}

We view \(\Gamma(\theta)\) as a subcomplex of the join
\(\bigast_{w\in W}Z_m(\FF)\), by placing the set \(D_w(\theta)\) in the
\(w\)-th copy of \(Z_m(\FF)\).  We claim that
\(|\Gamma(\theta)|\subseteq X\).

To view this, let \(\eta\) be a face of \(\Gamma(\theta)\).  Its set of sources is
a face \(I\in\mathcal{K}[U_\theta]\).  For each vertex of \(\eta\), say
\(z_w^{v_\tau}\), condition \eqref{eq:3-in-D-w} implies that this point lies in
\(D_w(u)\) for every \(u\in\theta\), because
\(\theta\subseteq \operatorname{st}(\tau)\subseteq O_{v_\tau}\).  The observation
above then implies that the corresponding tuple belongs to \(X_I\).  Therefore,
\(\eta\) is a face of \(X\).

Using \eqref{eq:4-connected}, we construct a \((\ZZ_2)^m\)-equivariant map
\[
        s:
        \skel_N V_m(\FF^{d+1})
        \longrightarrow X
        \tag{5}
\]
by the usual extension argument (i.e. inductively on the skeleta).  On vertices choose arbitrary
points in the corresponding set.  Suppose \(s\) has been constructed on
the \(a\)-skeleton, with \(a<N\), and let \(\theta\) be an \((a+1)\)-simplex.
By the inclusions \(\Gamma(\rho)\subseteq\Gamma(\theta)\) for
\(\rho\subseteq\theta\), we have \(s(\partial\theta)\subseteq
|\Gamma(\theta)|\).  Since \(|\Gamma(\theta)|\) is \((N-1)\)-connected and
\(a\le N-1\), the boundary map extends over \(\theta\).  Doing this over
free \((\ZZ_2)^m\)-orbits of simplices gives (5).

It remains to define the test map.  For
\[
        x=\sum_{w\in I} t_w z_w\in X
\]
with \(z_w=(z_{w,1},\dots,z_{w,m})\), define the map

\begin{align*}
    \Phi: X & \to (\FF^{r+1})^m \\
    \Phi(x)
        & =
        \left(
        \sum_{w\in I}
        t_w
        \bigl(z_{w,j}\varphi(w), z_{w,j}\bigr)
        \right)_{j=1}^m
\end{align*}
This map is \((\ZZ_2)^m\)-equivariant, where the group acts on the \(j\)-th
copy of \(\FF^{r+1}\) by the sign of the \(j\)-th coordinate.

We claim that \(0\notin\Phi(X)\).  Suppose otherwise.  Then for some
\(I\in\mathcal{K}\), some \(t_w>0\), and some \((z_w)_{w\in I}\in X_I\), we
have
\[
        \sum_{w\in I} t_w z_{w,j}=0,
        \qquad
        \sum_{w\in I} t_w z_{w,j}\varphi(w)=0
        \qquad
        \text{for }j=1,\dots,m.
\]
These are \(m=d-k\) affine \(\FF\)-dependencies of the points
\(\varphi(w)\), and they are not all trivial because each \(z_w\) has norm
one and the \(t_w\)'s are positive.  Since the dependency-modeling condition
holds on every face of \(\mathcal{K}\), as any face can be extended to a facet and using zero coefficients on the added vertices, there exist real numbers
\(\lambda_w\ge 0\) and points \(q_w\in F_w\) such that
\[
        \sum_{w\in I} \lambda_w t_w z_{w,j}=0,
        \qquad
        \sum_{w\in I} \lambda_w t_w z_{w,j}q_w=0
        \qquad
        \text{for }j=1,\dots,m,
\]
and the resulting system is nontrivial.  Setting \(b_w=\lambda_w t_w\)
contradicts the definition of \(X_I\).  Therefore \(0\notin\Phi(X)\).

Composing (5) with \(\Phi\) and then with radial projection gives a
\((\ZZ_2)^m\)-equivariant map
\[
        \skel_N V_m(\FF^{d+1})
        \longrightarrow
        S\bigl((\FF^{r+1})^m\bigr).
\]
This contradicts \cref{thm:stiefel-obstruction}, proving the theorem.
\end{proof}

As done in the previous sections, the selection structure theorem, in this case \cref{thm:selection-structure-sadovek} implies the matroidal version, \cite{sadovek2026colorful}*{Thm. 1.6}, if we use the selection structure from \cref{example:matroids}.

\subsection{Dolnikov's transversal theorem}

Just like Sadovek's result implies a matroidal version of Dolnikov's central transversal theorem \cite{Dolnikov1992}, we obtain a selection structure version of Dolnikov's theorem from \cref{thm:selection-structure-sadovek}.  In the introduction we describe the real-valued version, below is the full result.  Ultimately, the topological tool at the center of this proof is \cref{thm:stiefel-obstruction}.  For earlier deductions of the central transversal theorem and its generalizations via Borsuk--Ulam-type theorems for Stiefel manifolds, see \cites{Manta2024, McGinnis2026}.

\begin{theorem}[Selection structure Dolnikov transversal theorem]
\label{thm:selection-structure-dolnikov-full}
Let \(0\le r\le k<d\), let \(\FF\in\{\rr,\cc\}\), and
\(
        h=\dim_{\rr}(\FF)\,(d-k)+1\), \(
        N=\dim_{\rr}(\FF)\,(d-k)(r+1)\).
Let \(W\) be a finite set, let \( W=W_1\sqcup\cdots\sqcup W_{r+1}\)
be a partition, and let \( \Sigma=(\mathcal{K},\mathscr{L}) \)
be an \(N\)-admissible selection structure on \(W\).
For each \(w\in W\), let \(F_w\subseteq \FF^d\) be a nonempty compact
convex set.

Assume that for every \(i=1,\dots,r+1\) and every face
\(I\in\mathcal{K}\) such that
\( I\subseteq W_i\) and \( |I|\le h\),
we have
\( \bigcap_{w\in I}F_w\neq\emptyset\).

Then there exists \(W'\subseteq W\) such that the family \( \{F_w:w\in W'\}\)
has an \(\FF\)-\(k\)-transversal and
\( W\setminus W'\notin\mathscr{L}\).
\end{theorem}

\begin{proof}
We derive the theorem from \cref{thm:selection-structure-sadovek}.

Choose \(\FF\)-affinely independent points \( p_1,\dots,p_{r+1}\in \FF^r\).
Define the map
\[
        \varphi:W\longrightarrow \FF^r,
        \qquad
        \varphi(w)=p_i
        \quad\text{if }w\in W_i .
\]
We claim that \(\varphi\) models
\((d-k,\Sigma,\FF)\)-dependencies of
\(\{F_w:w\in W\}\).

Let $\mathcal{B} = \operatorname{Facets}(\mathcal{K})$ and \(B\in\mathcal B\).  Suppose that
\[
        \sum_{w\in B} a_{j,w}=0,
        \qquad
        \sum_{w\in B} a_{j,w}\varphi(w)=0,
        \qquad
        j=1,\dots,m,
\]
is a nontrivial system of \(m=d-k\) affine \(\FF\)-dependencies.
For each \(w\in B\), write
\(A_w=(a_{1,w},\dots,a_{m,w})\in \FF^m\).  Since \(p_1,\dots,p_{r+1}\) are affinely independent, for every
\(i=1,\dots,r+1\) we get
\(
        \sum_{w\in B\cap W_i} A_w=0
  \) in \(\FF^m\).
The original system is nontrivial, so there is some \(i_0\) for which the
vectors \(A_w\), \(w\in B\cap W_{i_0}\), are not all zero.

View \(\FF^m\) as a real vector space of dimension
\(\dim_{\rr}(\FF)(d-k)\). Since
\[
        \sum_{w\in B\cap W_{i_0}} A_w=0
\]
and not all of these vectors are zero, the origin lies in the convex hull of
the nonzero vectors among them. By Carathéodory's theorem, there is a subset
\( J\subseteq B\cap W_{i_0} \)
with \(|J|\le \dim_{\rr}(\FF)(d-k)+1=h \) and positive real numbers \(t_w>0\), \(w\in J\), such that
\[
        \sum_{w\in J} t_w A_w=0 .
\]
Because \(B\in\mathcal B\subseteq\mathcal{K}\) and \(\mathcal{K}\) is a
simplicial complex, \(J\in\mathcal{K}\). Also \(J\subseteq W_{i_0}\) and
\(|J|\le h\). By hypothesis, there exists a point
\(q\in \bigcap_{w\in J}F_w\).

Define
\[
        \lambda_w=
        \begin{cases}
        t_w, & w\in J,\\
        0, & w\notin J,
        \end{cases}
\]
choose \(q_w=q\) for \(w\in J\), and choose arbitrary \(q_w\in F_w\) for
\(w\notin J\). Then, for every \(j=1,\dots,m\),
\[
        \sum_{w\in B}\lambda_w a_{j,w}
        =
        \sum_{w\in J}t_w a_{j,w}
        =
        0,
\]
and
\[
        \sum_{w\in B}\lambda_w a_{j,w}q_w
        =
        q\sum_{w\in J}t_w a_{j,w}
        =
        0.
\]
This new system is nontrivial (i.e., not all coefficients are zero).  This means that \(\varphi\) models
\((d-k,\Sigma,\FF)\)-dependencies.

Now apply \cref{thm:selection-structure-sadovek} with \( N=\dim_{\rr}(\FF)(d-k)(r+1)\).
It gives us a subset \(W'\subseteq W\) such that
\(\{F_w:w\in W'\}\) has an \(\FF\)-\(k\)-transversal and
\(
        W\setminus W'\notin\mathscr{L}\),
as we wanted to show.
\end{proof}

\section{Remarks}\label{sec:remarks}

\subsection{On selection structure versions of Helly's theorem}\label{sec:remarks-helly}

This manuscript has two different proofs of the selection structure version of Helly's theorem, \cref{thm:selection-structure-helly}.

Our proof in \cref{sec:helly} applies to $d$-collapsible complexes.  This allows us to prove some selection-structure versions of Helly's theorem that are not covered by \cref{thm:selection-structure-sadovek}.

One example is a selection structure version of the Doignon--Bell--Scarf theorem \cites{Doignon1973, Bell1976/77, Scarf1977}, which is the equivalent version to Helly's theorem on the integer lattice.

\begin{theorem}[Selection structure Doignon--Bell--Scarf]\label{thm:selection-structure-doignon}
    Let $W$ be a finite set and $\mathcal{F}=\{F_w: w \in W\}$ be a family of convex sets in $\rr^d$ indexed by $W$, each with at least one point in $\zz^d$.  Let $\Sigma=(\mathcal{K}, \mathscr{L})$ be a $(2^d-1)$-admissible selection structure on $W$.  Suppose that for every $W' \subseteq W$ such that $W \setminus W' \not\in \mathscr{L}$ we have $\bigcap_{w\in  W'} F_w$ has no integer points.  Then, there exists $I \in \mathcal K$ such that $\bigcap_{w \in I}F_w$ has no integer points.
\end{theorem}

The specialization to matroids of the Doignon--Bell--Scarf theorem is also new.  Other versions of Helly's theorem are proved via a bound on the collapsibility of the nerve complex, such as the recent results on discrete Helly theorems for axis-parallel boxes \cite{Edwards2025}.

\begin{proof}
    The nerve complex of the family of sets $\zz^d \cap F_w$ for $w \in W$ is $(2^d-1)$-collapsible.  This is a direct consequence of De Loera et al.'s proof of the colorful Doignon--Bell--Scarf theorem \cite{DeLoera2017}.  We apply \cref{thm:selection-structure-d-collapsible} to the nerve to conclude.
\end{proof}

\subsection{Direct comparison with matroids}\label{sec:improving-selection-structures}

The selection structures in \cref{sec:chessboard-matching-choice} do not come from matroids.  The complex $\Delta_{2,3}$ is graph: a cycle of length $6$.  If we apply either \cref{thm:chessboard-selection-structures} or \cref{ex:selection-structure-graph}, we get the same selection structure.  The set $\mathscr{L}$ must contain sets that contain a path with four vertices.

One point of concern might be that, even if a selection structure does not come from a matroid, it might still be possible for some matroid $M$ to imply the same applications.  In this section, we show how the case $\Delta_{2,3}$ with $\mathscr{L}$ as described above gives applications that do not follow from matroidal versions.  Let us denote this selection structure by $\Sigma_{2,3}$.

The basis graph of a rank-\(2\) matroid is a complete multipartite
graph on the nonloop elements: two elements form a basis if and only if
they lie in different parallel classes of the matroid.  If a matroid were to give the
same implications as $\Sigma_{2,3}$, every basis would have
to be a non-attacking pair, so the basis graph would have to be a subgraph of
the \(C_6\) graph \(\Delta_{2,3}\).

But a complete multipartite graph contained in \(C_6\) is either a single
edge or a two-edge star \(K_{1,2}\), up to isolated loop elements.  In either
case there is a \(5\)-element subset \(S\subseteq W\) containing no basis of
the matroid.

Therefore every rank-\(2\) matroid whose bases are non-attacking pairs has some
\(5\)-element set \(S\) of rank at most \(1\).  Since \(S\in\mathscr{L}\), the
selection-structure applications theorem would not require the covering condition for
\(W'=W\setminus S\).  The matroidal applications with $M$ would require it.

\section{Acknowledgments}

The author used ChatGPT 5.5 for simplification of proofs and for improvements for selection structures arising from chessboard complexes.  The author verified and edited all useful output and takes full responsibility for the accuracy and integrity of the manuscript.

\bibliography{refref}

@Article{Chan2020,
  author     = {Chan, Yu Hin and Chen, Shujian and Frick, Florian and Hull, J. Tristan},
  journal    = {Topol. Methods Nonlinear Anal.},
  title      = {Borsuk-{U}lam theorems for products of spheres and {S}tiefel manifolds revisited},
  year       = {2020},
  issn       = {1230-3429},
  number     = {2},
  pages      = {553--564},
  volume     = {55},
  doi        = {10.12775/tmna.2019.103},
  fjournal   = {Topological Methods in Nonlinear Analysis},
  mrclass    = {55M20 (54H25)},
  mrnumber   = {4131166},
  mrreviewer = {Tha\'is\ Fernanda Mendes Monis},
  url        = {https://doi.org/10.12775/tmna.2019.103},
}

@Article{Sadovek2026,
  author   = {Sadovek, Nikola and Sober\'on, Pablo},
  journal  = {Trans. Amer. Math. Soc.},
  title    = {Complex analogues of the {T}verberg--{V}re\'cica conjecture and central transversal theorems},
  year     = {2026},
  issn     = {0002-9947,1088-6850},
  number   = {5},
  pages    = {3665--3691},
  volume   = {379},
  doi      = {10.1090/tran/9555},
  fjournal = {Transactions of the American Mathematical Society},
  mrclass  = {52A35 (55N91 55R91)},
  mrnumber = {5060442},
  url      = {https://doi.org/10.1090/tran/9555},
}

@Article{Bjoerner1994,
  author     = {Bj\"orner, A. and Lov\'asz, L. and Vre\'cica, S. T. and {\v{Z}}ivaljevi\'c, R. T.},
  journal    = {J. London Math. Soc. (2)},
  title      = {Chessboard complexes and matching complexes},
  year       = {1994},
  issn       = {0024-6107,1469-7750},
  number     = {1},
  pages      = {25--39},
  volume     = {49},
  doi        = {10.1112/jlms/49.1.25},
  fjournal   = {Journal of the London Mathematical Society. Second Series},
  mrclass    = {52B20 (05C70 13F50 55U10)},
  mrnumber   = {1253009},
  mrreviewer = {Rafael\ H.\ Villarreal},
  url        = {https://doi.org/10.1112/jlms/49.1.25},
}

@Book{Matousek2002,
  author     = {Matou{\v{s}}ek, Ji{\v{r}}{\'{i}}},
  publisher  = {Springer-Verlag, New York},
  title      = {Lectures on discrete geometry},
  year       = {2002},
  isbn       = {0-387-95373-6},
  series     = {Graduate Texts in Mathematics},
  volume     = {212},
  doi        = {10.1007/978-1-4613-0039-7},
  mrclass    = {52Cxx (52-01)},
  mrnumber   = {1899299},
  mrreviewer = {E.\ Hertel},
  pages      = {xvi+481},
  url        = {https://doi.org/10.1007/978-1-4613-0039-7},
}

@Article{Zivaljevic1990,
  author     = {{\v{Z}}ivaljevi{\'{c}}, Rade T. and Vre{\'{c}}ica, Sini{\v{s}}a T.},
  journal    = {Bull. London Math. Soc.},
  title      = {An extension of the ham sandwich theorem},
  year       = {1990},
  issn       = {0024-6093,1469-2120},
  number     = {2},
  pages      = {183--186},
  volume     = {22},
  doi        = {10.1112/blms/22.2.183},
  fjournal   = {The Bulletin of the London Mathematical Society},
  mrclass    = {52A37 (55R10)},
  mrnumber   = {1045292},
  mrreviewer = {S.\ Y.\ Husseini},
  url        = {https://doi.org/10.1112/blms/22.2.183},
}

@Article{Dolnikov1992,
  author     = {Dolnikov, Vladimir L.},
  journal    = {Math. Notes},
  title      = {A generalization of the sandwich theorem},
  year       = {1992},
  number     = {1-2},
  pages      = {7771--779},
  volume     = {52},
  doi        = {10.1007/BF01236771},
  fjournal   = {Mathematical Notes},
  mrclass    = {28A12 (52A37)},
  mrnumber   = {1187871},
  mrreviewer = {B\'ela\ Uhrin},
  url        = {https://doi.org/10.1007/BF01236771},
}

@Article{DeLoera2017,
  author     = {De Loera, Jes\'us A. and La Haye, Reuben N. and Oliveros, D\'eborah and Rold\'an-Pensado, Edgardo},
  journal    = {Adv. Geom.},
  title      = {Helly numbers of algebraic subsets of {$\Bbb R^d$} and an extension of {D}oignon's theorem},
  year       = {2017},
  issn       = {1615-715X,1615-7168},
  number     = {4},
  pages      = {473--482},
  volume     = {17},
  doi        = {10.1515/advgeom-2017-0028},
  fjournal   = {Advances in Geometry},
  mrclass    = {52A35 (52C07)},
  mrnumber   = {3714450},
  mrreviewer = {Konrad\ J.\ Swanepoel},
  url        = {https://doi.org/10.1515/advgeom-2017-0028},
}

@Article{Scarf1977,
  author     = {Scarf, Herbert E.},
  journal    = {Proc. Nat. Acad. Sci. U.S.A.},
  title      = {An observation on the structure of production sets with indivisibilities},
  year       = {1977},
  issn       = {0027-8424},
  number     = {9},
  pages      = {3637--3641},
  volume     = {74},
  doi        = {10.1073/pnas.74.9.3637},
  fjournal   = {Proceedings of the National Academy of Sciences of the United States of America},
  mrclass    = {90C10},
  mrreviewer = {F.\ Giannessi},
  url        = {https://doi.org/10.1073/pnas.74.9.3637},
}

@Article{Bell1976/77,
  author     = {Bell, David E.},
  journal    = {Studies in Appl. Math.},
  title      = {A theorem concerning the integer lattice},
  year       = {1976/77},
  issn       = {0022-2526,1467-9590},
  number     = {2},
  pages      = {187--188},
  volume     = {56},
  doi        = {10.1002/sapm1977562187},
  fjournal   = {Studies in Applied Mathematics},
  mrclass    = {90C99},
  mrreviewer = {George\ E.\ Andrews},
  url        = {https://doi.org/10.1002/sapm1977562187},
}

@Article{Doignon1973,
  author     = {Doignon, Jean-Paul},
  journal    = {J. Geom.},
  title      = {Convexity in cristallographical lattices},
  year       = {1973},
  issn       = {0047-2468,1420-8997},
  pages      = {71--85},
  volume     = {3},
  doi        = {10.1007/BF01949705},
  fjournal   = {Journal of Geometry},
  mrclass    = {05B35 (52A35)},
  mrreviewer = {G.\ L.\ Alexanderson},
  url        = {https://doi.org/10.1007/BF01949705},
}

@Article{Vrecica2011,
  author   = {Vre\'cica, Sini\v sa T. and {\v{ Z}}ivaljevi\'c, Rade T.},
  journal  = {J. Combin. Theory Ser. A},
  title    = {Chessboard complexes indomitable},
  year     = {2011},
  issn     = {0097-3165,1096-0899},
  number   = {7},
  pages    = {2157--2166},
  volume   = {118},
  doi      = {10.1016/j.jcta.2011.04.007},
  fjournal = {Journal of Combinatorial Theory. Series A},
  mrclass  = {52B99 (05E45)},
  mrnumber = {2802193},
  url      = {https://doi.org/10.1016/j.jcta.2011.04.007},
}

@article{blagojevic2026cone,
  title={Cone and constrained colorful Carath$\backslash$'eodory Theorems},
  author={Blagojevi{\'c}, Pavle VM and Karasev, Roman N and Zsigri, B{\'a}lint},
  journal={arXiv preprint arXiv:2607.06172},
  year={2026}
}

@Article{Biermann2015,
  author   = {Biermann, Jennifer and Francisco, Christopher A. and H\`a, Huy T\`ai and Van Tuyl, Adam},
  journal  = {J. Commut. Algebra},
  title    = {Partial coloring, vertex decomposability, and sequentially {C}ohen-{M}acaulay simplicial complexes},
  year     = {2015},
  issn     = {1939-0807,1939-2346},
  number   = {3},
  pages    = {337--352},
  volume   = {7},
  doi      = {10.1216/JCA-2015-7-3-337},
  fjournal = {Journal of Commutative Algebra},
  mrclass  = {05E45 (05A15 13F55)},
  mrnumber = {3433985},
  url      = {https://doi.org/10.1216/JCA-2015-7-3-337},
}

@Article{Provan1980,
  author     = {Provan, J. Scott and Billera, Louis J.},
  journal    = {Math. Oper. Res.},
  title      = {Decompositions of simplicial complexes related to diameters of convex polyhedra},
  year       = {1980},
  issn       = {0364-765X,1526-5471},
  number     = {4},
  pages      = {576--594},
  volume     = {5},
  doi        = {10.1287/moor.5.4.576},
  fjournal   = {Mathematics of Operations Research},
  mrclass    = {52A25 (90C05)},
  mrnumber   = {593648},
  mrreviewer = {J.\ Parida},
  url        = {https://doi.org/10.1287/moor.5.4.576},
}

@Article{Ziegler1994,
  author     = {Ziegler, G\"unter M.},
  journal    = {Israel J. Math.},
  title      = {Shellability of chessboard complexes},
  year       = {1994},
  issn       = {0021-2172,1565-8511},
  number     = {1-3},
  pages      = {97--110},
  volume     = {87},
  doi        = {10.1007/BF02772986},
  fjournal   = {Israel Journal of Mathematics},
  mrclass    = {06A08 (05B35)},
  mrnumber   = {1286818},
  mrreviewer = {Rade\ \v Zivaljevi\'c},
  url        = {https://doi.org/10.1007/BF02772986},
}

@Book{Barany2021,
  author     = {B\'ar\'any, Imre},
  publisher  = {American Mathematical Society, Providence, RI},
  title      = {Combinatorial convexity},
  year       = {2021},
  isbn       = {978-1-4704-6709-8},
  series     = {University Lecture Series},
  volume     = {77},
  doi        = {10.1090/ulect/077},
  mrclass    = {52A35},
  mrnumber   = {4390801},
  mrreviewer = {Jen\H o\ Lehel},
  pages      = {viii+148},
  url        = {https://doi.org/10.1090/ulect/077},
}

@Article{DeLoera2019,
  author     = {De Loera, Jes\'us A. and Goaoc, Xavier and Meunier, Fr\'ed\'eric and Mustafa, Nabil H.},
  journal    = {Bull. Amer. Math. Soc. (N.S.)},
  title      = {The discrete yet ubiquitous theorems of {C}arath\'eodory, {H}elly, {S}perner, {T}ucker, and {T}verberg},
  year       = {2019},
  issn       = {0273-0979,1088-9485},
  number     = {3},
  pages      = {415--511},
  volume     = {56},
  doi        = {10.1090/bull/1653},
  fjournal   = {American Mathematical Society. Bulletin. New Series},
  mrclass    = {52A35 (57M99 90C25 91A80 91B32)},
  mrreviewer = {Mircea\ Balaj},
  url        = {https://doi.org/10.1090/bull/1653},
}

@Article{McGinnis2024a,
  author     = {McGinnis, Daniel and Zerbib, Shira},
  journal    = {Discrete Comput. Geom.},
  title      = {A sparse colorful polytopal {KKM} theorem},
  year       = {2024},
  issn       = {0179-5376,1432-0444},
  number     = {3},
  pages      = {945--959},
  volume     = {71},
  doi        = {10.1007/s00454-022-00464-y},
  fjournal   = {Discrete \& Computational Geometry. An International Journal of Mathematics and Computer Science},
  mrclass    = {52A35},
  mrnumber   = {4716255},
  mrreviewer = {Gabriela\ Cristescu},
  url        = {https://doi.org/10.1007/s00454-022-00464-y},
}

@InCollection{Holmsen:2017uf,
  author    = {Holmsen, Andreas F. and Wenger, Rephael},
  publisher = {Chapman and Hall/CRC},
  title     = {{Helly-type theorems and geometric transversals}},
  year      = {2017},
  edition   = {3rd},
  pages     = {91--123},
  series    = {Handbook of Discrete and Computational Geometry},
}

@article{sadovek2026colorful,
  title={On colorful generalizations of the Goodman--Pollack transversal problem},
  author={Sadovek, Nikola},
  journal={arXiv preprint arXiv:2604.19644},
  year={2026}
}

@article{blagojevic2025colorful,
  title={A Colorful Version of Carath{\'e}odory's Theorem plus a constraint},
  author={Blagojevic, Pavle VM},
  journal={arXiv e-prints},
  pages={arXiv--2509},
  year={2025}
}

@Article{Edwards2025,
  author     = {Edwards, Timothy and Sober\'on, Pablo},
  journal    = {SIAM J. Discrete Math.},
  title      = {Extensions of discrete {H}elly theorems for boxes},
  year       = {2025},
  issn       = {0895-4801,1095-7146},
  number     = {2},
  pages      = {1349--1362},
  volume     = {39},
  doi        = {10.1137/24M1658358},
  fjournal   = {SIAM Journal on Discrete Mathematics},
  mrclass    = {52A35},
  mrnumber   = {4922382},
  mrreviewer = {Mircea\ Balaj},
  url        = {https://doi.org/10.1137/24M1658358},
}

@InCollection{Barany2017,
  author    = {B\'ar\'any, Imre and Kalai, Gil and Meshulam, Roy},
  booktitle = {A journey through discrete mathematics},
  publisher = {Springer, Cham},
  title     = {A {T}verberg type theorem for matroids},
  year      = {2017},
  isbn      = {978-3-319-44478-9; 978-3-319-44479-6},
  pages     = {115--121},
  mrclass   = {05B35},
  mrnumber  = {3726596},
}

@Article{Cheong2024,
  author   = {Cheong, Otfried and Goaoc, Xavier and Holmsen, Andreas F.},
  journal  = {Discrete Comput. Geom.},
  title    = {Some new results on geometric transversals},
  year     = {2024},
  issn     = {0179-5376,1432-0444},
  number   = {2},
  pages    = {674--703},
  volume   = {72},
  doi      = {10.1007/s00454-023-00573-2},
  fjournal = {Discrete \& Computational Geometry. An International Journal of Mathematics and Computer Science},
  mrclass  = {52A35 (52C35)},
  mrnumber = {4800737},
  url      = {https://doi.org/10.1007/s00454-023-00573-2},
}

@Article{Holmsen2022,
  author  = {Holmsen, Andreas F},
  journal = {arXiv preprint arXiv:2205.04077},
  title   = {{A colorful Goodman-Pollack-Wenger theorem}},
  year    = {2022},
}

@Article{Su1999,
  author     = {Su, Francis Edward},
  journal    = {Amer. Math. Monthly},
  title      = {Rental harmony: {S}perner's lemma in fair division},
  year       = {1999},
  issn       = {0002-9890,1930-0972},
  number     = {10},
  pages      = {930--942},
  volume     = {106},
  doi        = {10.2307/2589747},
  fjournal   = {American Mathematical Monthly},
  mrclass    = {91B08 (05D05)},
  mrnumber   = {1732499},
  mrreviewer = {Konrad\ Engel},
  url        = {https://doi.org/10.2307/2589747},
}

@Article{Manta2024,
  author     = {Manta, Michael N. and Sober\'on, Pablo},
  journal    = {Discrete Comput. Geom.},
  title      = {Generalizations of the {Y}ao-{Y}ao partition theorem and central transversal theorems},
  year       = {2024},
  issn       = {0179-5376,1432-0444},
  number     = {4},
  pages      = {1381--1402},
  volume     = {71},
  doi        = {10.1007/s00454-023-00536-7},
  fjournal   = {Discrete \& Computational Geometry. An International Journal of Mathematics and Computer Science},
  mrclass    = {52A35},
  mrnumber   = {4742209},
  mrreviewer = {Gabriela\ Cristescu},
  url        = {https://doi.org/10.1007/s00454-023-00536-7},
}

@Article{Kim2024,
  author     = {Kim, Minki and Lew, Alan},
  journal    = {Forum Math. Sigma},
  title      = {Extensions of the colorful {H}elly theorem for {$d$}-collapsible and {$d$}-{L}eray complexes},
  year       = {2024},
  issn       = {2050-5094},
  pages      = {Paper No. e44, 14},
  volume     = {12},
  doi        = {10.1017/fms.2024.23},
  fjournal   = {Forum of Mathematics. Sigma},
  mrclass    = {52A35 (05E45 55U10)},
  mrnumber   = {4726502},
  mrreviewer = {L\'aszl\'o\ Szab\'o},
  url        = {https://doi.org/10.1017/fms.2024.23},
}

@InCollection{Bjoerner1992,
  author     = {Bj\"orner, Anders},
  booktitle  = {Matroid applications},
  publisher  = {Cambridge Univ. Press, Cambridge},
  title      = {The homology and shellability of matroids and geometric lattices},
  year       = {1992},
  isbn       = {0-521-38165-7},
  pages      = {226--283},
  series     = {Encyclopedia Math. Appl.},
  volume     = {40},
  doi        = {10.1017/CBO9780511662041.008},
  mrclass    = {52B40 (05B35 55N99)},
  mrnumber   = {1165544},
  mrreviewer = {Michel\ Yves\ Jambu},
  url        = {https://doi.org/10.1017/CBO9780511662041.008},
}

@Article{McGinnis2026,
  author   = {McGinnis, Daniel and Sadovek, Nikola},
  journal  = {Adv. Math.},
  title    = {A necessary and sufficient condition for {$k$}-transversals},
  year     = {2026},
  issn     = {0001-8708,1090-2082},
  pages    = {Paper No. 110829, 13},
  volume   = {490},
  doi      = {10.1016/j.aim.2026.110829},
  fjournal = {Advances in Mathematics},
  mrclass  = {52A35},
  url      = {https://doi.org/10.1016/j.aim.2026.110829},
}

@Article{Wenger1990,
  author     = {Wenger, Rephael},
  journal    = {Discrete Comput. Geom.},
  title      = {A generalization of {H}adwiger's transversal theorem to intersecting sets},
  year       = {1990},
  issn       = {0179-5376,1432-0444},
  number     = {4},
  pages      = {383--388},
  volume     = {5},
  doi        = {10.1007/BF02187799},
  fjournal   = {Discrete \& Computational Geometry. An International Journal of Mathematics and Computer Science},
  mrclass    = {52A35 (52A10)},
  mrnumber   = {1043720},
  mrreviewer = {Marilyn\ Breen},
  url        = {https://doi.org/10.1007/BF02187799},
}

@Article{Pollack1990,
  author     = {Pollack, Richard and Wenger, Rephael},
  journal    = {Combinatorica},
  title      = {Necessary and sufficient conditions for hyperplane transversals},
  year       = {1990},
  issn       = {0209-9683},
  number     = {3},
  pages      = {307--311},
  volume     = {10},
  doi        = {10.1007/BF02122783},
  fjournal   = {Combinatorica. An International Journal on Combinatorics and the Theory of Computing},
  mrclass    = {52A35},
  mrnumber   = {1092546},
  mrreviewer = {Marilyn\ Breen},
  url        = {https://doi.org/10.1007/BF02122783},
}

@Article{Goodman1988,
  author     = {Goodman, Jacob E. and Pollack, Richard},
  journal    = {J. Amer. Math. Soc.},
  title      = {Hadwiger's transversal theorem in higher dimensions},
  year       = {1988},
  issn       = {0894-0347,1088-6834},
  number     = {2},
  pages      = {301--309},
  volume     = {1},
  doi        = {10.2307/1990918},
  fjournal   = {Journal of the American Mathematical Society},
  mrclass    = {52A20 (05B35 52A35)},
  mrnumber   = {928260},
  mrreviewer = {Marilyn\ Breen},
  url        = {https://doi.org/10.2307/1990918},
}

@Article{Sperner1928,
  author   = {Sperner, Emanuel},
  journal  = {Abh. Math. Sem. Univ. Hamburg},
  title    = {Neuer beweis f\"ur die invarianz der dimensionszahl und des gebietes},
  year     = {1928},
  issn     = {0025-5858,1865-8784},
  number   = {1},
  pages    = {265--272},
  volume   = {6},
  doi      = {10.1007/BF02940617},
  fjournal = {Abhandlungen aus dem Mathematischen Seminar der Universit\"at Hamburg},
  mrclass  = {99-04},
  mrnumber = {3069504},
  url      = {https://doi.org/10.1007/BF02940617},
}

@Book{Brams:1996wt,
  author    = {Brams, Steven J. and Taylor, Alan D.},
  publisher = {Cambridge University Press},
  title     = {{Fair Division: From cake-cutting to dispute resolution}},
  year      = {1996},
  series    = {Cambridge University Press},
}

@InCollection{Amenta2017,
  author     = {Amenta, Nina and De Loera, Jes\'us A. and Sober\'on, Pablo},
  booktitle  = {Algebraic and geometric methods in discrete mathematics},
  publisher  = {Amer. Math. Soc., Providence, RI},
  title      = {Helly's theorem: new variations and applications},
  year       = {2017},
  isbn       = {978-1-4704-2321-6},
  pages      = {55--95},
  series     = {Contemp. Math.},
  volume     = {685},
  doi        = {10.1090/conm/685},
  mrclass    = {52A35},
  mrreviewer = {Jen\H o\ Lehel},
  url        = {https://doi.org/10.1090/conm/685},
}

@InCollection{Goodman1993,
  author    = {Goodman, Jacob E. and Pollack, Richard and Wenger, Rephael},
  booktitle = {New trends in discrete and computational geometry},
  publisher = {Springer, Berlin},
  title     = {Geometric transversal theory},
  year      = {1993},
  isbn      = {3-540-55713-X},
  pages     = {163--198},
  series    = {Algorithms Combin.},
  volume    = {10},
  doi       = {10.1007/978-3-642-58043-7\_8},
  mrclass   = {52A35 (52B55 68U05)},
  mrnumber  = {1228043},
  url       = {https://doi.org/10.1007/978-3-642-58043-7_8},
}

@InCollection{Tancer2013,
  author    = {Tancer, Martin},
  booktitle = {Thirty essays on geometric graph theory},
  publisher = {Springer, New York},
  title     = {Intersection patterns of convex sets via simplicial complexes: a survey},
  year      = {2013},
  isbn      = {978-1-4614-0109-4; 978-1-4614-0110-0},
  pages     = {521--540},
  doi       = {10.1007/978-1-4614-0110-0\_28},
  mrclass   = {52A35 (05E45 52A20)},
  mrnumber  = {3205172},
  url       = {https://doi.org/10.1007/978-1-4614-0110-0_28},
}

@Article{Caratheodory1907,
  author   = {Carath\'eodory, Constantin},
  journal  = {Math. Ann.},
  title    = {\"Uber den {V}ariabilit\"atsbereich der {K}oeffizienten von {P}otenzreihen, die gegebene {W}erte nicht annehmen},
  year     = {1907},
  issn     = {0025-5831,1432-1807},
  number   = {1},
  pages    = {95--115},
  volume   = {64},
  doi      = {10.1007/BF01449883},
  fjournal = {Mathematische Annalen},
  mrclass  = {99-04},
  mrnumber = {1511425},
  url      = {https://doi.org/10.1007/BF01449883},
}

@Article{Helly:1923wr,
  author  = {Helly, Eduard},
  journal = {Jahresbericht der Deutschen Mathematiker-Vereinigung},
  title   = {{\"Uber Mengen konvexer Körper mit gemeinschaftlichen Punkte.}},
  year    = {1923},
  pages   = {175--176},
  volume  = {32},
}

@Article{Komiya1994,
  author     = {Komiya, Hidetoshi},
  journal    = {Econom. Theory},
  title      = {A simple proof of {K}-{K}-{M}-{S} theorem},
  year       = {1994},
  issn       = {0938-2259,1432-0479},
  number     = {3},
  pages      = {463--466},
  volume     = {4},
  doi        = {10.1007/BF01215383},
  fjournal   = {Economic Theory},
  mrclass    = {90A14 (90D12)},
  mrnumber   = {1279465},
  mrreviewer = {F.\ Patrone},
  url        = {https://doi.org/10.1007/BF01215383},
}

@Article{Wegner1975,
  author     = {Wegner, Gerd},
  journal    = {Arch. Math. (Basel)},
  title      = {{$d$}-collapsing and nerves of families of convex sets},
  year       = {1975},
  issn       = {0003-889X,1420-8938},
  pages      = {317--321},
  volume     = {26},
  doi        = {10.1007/BF01229745},
  fjournal   = {Archiv der Mathematik},
  mrclass    = {57C05 (52A05)},
  mrnumber   = {375333},
  mrreviewer = {N.\ Pol\'akov\'a},
  url        = {https://doi.org/10.1007/BF01229745},
}

@Article{Kalai2005,
  author     = {Kalai, Gil and Meshulam, Roy},
  journal    = {Adv. Math.},
  title      = {A topological colorful {H}elly theorem},
  year       = {2005},
  issn       = {0001-8708,1090-2082},
  number     = {2},
  pages      = {305--311},
  volume     = {191},
  doi        = {10.1016/j.aim.2004.03.009},
  fjournal   = {Advances in Mathematics},
  mrclass    = {52A35 (55N35)},
  mrnumber   = {2103215},
  mrreviewer = {Yu.\ A.\ Shashkin},
  url        = {https://doi.org/10.1016/j.aim.2004.03.009},
}

@Article{Frick2019,
  author     = {Frick, Florian and Zerbib, Shira},
  journal    = {Combinatorica},
  title      = {Colorful coverings of polytopes and piercing numbers of colorful {$d$}-intervals},
  year       = {2019},
  issn       = {0209-9683,1439-6912},
  number     = {3},
  pages      = {627--637},
  volume     = {39},
  doi        = {10.1007/s00493-018-3891-1},
  fjournal   = {Combinatorica. An International Journal on Combinatorics and the Theory of Computing},
  mrclass    = {52A35 (05B40 52B11 55M20)},
  mrnumber   = {3989263},
  mrreviewer = {Konrad\ J.\ Swanepoel},
  url        = {https://doi.org/10.1007/s00493-018-3891-1},
}

@article{Mcginnis2024survey,
  title={{Using the KKM theorem}},
  author={McGinnis, Daniel and Zerbib, Shira},
  journal={arXiv preprint arXiv:2408.03921},
  year={2024}
}

@Article{Asada2018,
  author     = {Asada, Megumi and Frick, Florian and Pisharody, Vivek and Polevy, Maxwell and Stoner, David and Tsang, Ling Hei and Wellner, Zoe},
  journal    = {SIAM J. Discrete Math.},
  title      = {Fair division and generalizations of {S}perner- and {KKM}-type results},
  year       = {2018},
  issn       = {0895-4801,1095-7146},
  number     = {1},
  pages      = {591--610},
  volume     = {32},
  doi        = {10.1137/17M1116210},
  fjournal   = {SIAM Journal on Discrete Mathematics},
  mrclass    = {91B32 (52A35 54H25)},
  mrnumber   = {3769696},
  mrreviewer = {Lei\ Deng},
  url        = {https://doi.org/10.1137/17M1116210},
}

@article{McGinnis2024-arxiv,
  title={{Matroid colorings of KKM covers}},
  author={McGinnis, Daniel},
  journal={arXiv preprint arXiv:2409.03026},
  year={2024}
}

@Article{Soberon2022,
  author   = {Sober\'{o}n, Pablo},
  journal  = {Proc. Amer. Math. Soc. Ser. B},
  title    = {Fair distributions for more participants than allocations},
  year     = {2022},
  issn     = {2330-1511},
  pages    = {404--414},
  volume   = {9},
  doi      = {10.1090/bproc/129},
  fjournal = {Proceedings of the American Mathematical Society. Series B},
  mrclass  = {91B32},
  url      = {https://doi.org/10.1090/bproc/129},
}

@Article{Gale1984,
  author   = {Gale, David},
  journal  = {Internat. J. Game Theory},
  title    = {Equilibrium in a discrete exchange economy with money},
  year     = {1984},
  issn     = {0020-7276},
  number   = {1},
  pages    = {61--64},
  volume   = {13},
  doi      = {10.1007/BF01769865},
  fjournal = {International Journal of Game Theory},
  mrclass  = {90A14},
  url      = {https://doi.org/10.1007/BF01769865},
}

@Article{Knaster:1929vi,
  author  = {Knaster, Bronis{\l}aw and Kuratowski, Kazimierz and Mazurkiewicz, Stefan},
  journal = {Fundamenta Mathematicae},
  title   = {{Ein Beweis des Fixpunktsatzes f{\"u}r n-dimensionale Simplexe}},
  year    = {1929},
  number  = {1},
  pages   = {132--137},
  volume  = {14},
}

@Article{Shih1993,
  author     = {Shih, Mau-Hsiang and Lee, Shyh Nan},
  journal    = {Math. Ann.},
  title      = {Combinatorial formulae for multiple set-valued labellings},
  year       = {1993},
  issn       = {0025-5831,1432-1807},
  number     = {1},
  pages      = {35--61},
  volume     = {296},
  doi        = {10.1007/BF01445093},
  fjournal   = {Mathematische Annalen},
  mrclass    = {05A05 (05A15 52B05)},
  mrnumber   = {1213370},
  mrreviewer = {George\ M.\ Rassias},
  url        = {https://doi.org/10.1007/BF01445093},
}

@InCollection{Shapley1973,
  author     = {Shapley, Lloyd S.},
  booktitle  = {Mathematical programming ({P}roc. {A}dvanced {S}em., {U}niv. {W}isconsin, {M}adison, {W}is., 1972)},
  publisher  = {Academic Press, New York-London},
  title      = {On balanced games without side payments},
  year       = {1973},
  pages      = {261--290},
  mrclass    = {90D40},
  mrnumber   = {389244},
  mrreviewer = {Louis\ J.\ Billera},
}

@Article{Holmsen2016,
  author     = {Holmsen, Andreas F.},
  journal    = {Adv. Math.},
  title      = {The intersection of a matroid and an oriented matroid},
  year       = {2016},
  issn       = {0001-8708,1090-2082},
  pages      = {1--14},
  volume     = {290},
  doi        = {10.1016/j.aim.2015.11.040},
  fjournal   = {Advances in Mathematics},
  mrclass    = {52A35 (05B35 52C40 55U10)},
  mrnumber   = {3451916},
  mrreviewer = {Winfried\ Hochst\"attler},
  url        = {https://doi.org/10.1016/j.aim.2015.11.040},
}

@Article{Barany1982,
  author     = {B\'ar\'any, Imre},
  journal    = {Discrete Math.},
  title      = {A generalization of {C}arath\'eodory's theorem},
  year       = {1982},
  issn       = {0012-365X,1872-681X},
  number     = {2-3},
  pages      = {141--152},
  volume     = {40},
  doi        = {10.1016/0012-365X(82)90115-7},
  fjournal   = {Discrete Mathematics},
  mrclass    = {52A35},
  mrreviewer = {Marilyn\ Breen},
  url        = {https://doi.org/10.1016/0012-365X(82)90115-7},
}

\end{document}